\documentclass[a4paper]{amsart}
\usepackage{color}
\usepackage[latin1]{inputenc}
\usepackage[T1]{fontenc}
\usepackage{amsfonts}
\usepackage{amssymb}
\usepackage{amsmath}
\usepackage{amsthm}
\usepackage{graphicx,type1cm,eso-pic,color}
\usepackage{float}
\newtheorem{theorem}{Theorem}[section]
\newtheorem{lemma}[theorem]{Lemma}
\newtheorem{proposition}[theorem]{Proposition}
\newtheorem{corollary}[theorem]{Corollary}

\newtheorem{assumption}[theorem]{Assumption}

\begin{document}
\setlength\arraycolsep{2pt}
\title{Bandwidth Selection for the Wolverton-Wagner Estimator}
\author{Fabienne COMTE*}
\address{*Laboratoire MAP5, Universit\'e Paris Descartes, Paris, France}
\email{fabienne.comte@parisdescartes.fr}
\author{Nicolas MARIE**}
\address{**Laboratoire Modal'X, Universit\'e Paris Nanterre, Nanterre, France}
\email{nmarie@parisnanterre.fr}
\address{**ESME Sudria, Paris, France}
\email{nicolas.marie@esme.fr}
\date{}
\maketitle
\noindent
%


%
\begin{abstract}
For $n$ independent random variables having the same H\"older continuous density, this paper deals with controls of the Wolverton-Wagner's estimator MSE and MISE. Then, for a bandwidth $h_n(\beta)$, estimators of $\beta$ are obtained by a Goldenshluger-Lepski type method and a Lacour-Massart-Rivoirard type method. Some numerical experiments are provided for this last method.
\end{abstract}
%
\tableofcontents
%


%
\section{Introduction}
Consider $n\in\mathbb N^*$ independent random variables $X_1,\dots,X_n$ having the same probability distribution of density $f$ with respect to Lebesgue's measure.
\\
\\
The usual Parzen \cite{Par} - Rosenblatt \cite{Ros} kernel estimator of $f$ is defined by
\begin{displaymath}
\widehat f_{n,h}(x) :=
\frac{1}{nh}\sum_{k = 1}^{n}K\left(\frac{X_k - x}{h}\right)
\textrm{ $;$ }x\in\mathbb R,
\end{displaymath}
where $h > 0$ and $K :\mathbb R\rightarrow\mathbb R_+$ is a kernel. In 1969, Wolverton and Wagner introduced in \cite{WW69} a variant of $\widehat f_{n,h}(x)$ defined by
\begin{equation}\label{WW}
\widehat f_{n,\mathbf h_n}(x) :=
\frac{1}{n}\sum_{k = 1}^{n}\frac{1}{h_k}K\left(\frac{X_k - x}{h_k}\right),
\end{equation}
where $\mathbf h_n = (h_1,\dots,h_n)$ and $0 < h_n <\dots < h_1$. Thanks to its recursive form, this type of estimator is well-suited to online treatment of data: by denoting $\mathbf h_{n+1}=(h_1, \dots, h_n, h_{n+1})$,
\begin{displaymath}
\widehat f_{n+1,\mathbf h_{n+1}}(x) = \frac n{n+1} \widehat f_{n,\mathbf h_n}(x)
+ \frac{1}{(n+1) h_{n+1}} K\left(\frac{X_{n+1}-x}{h_{n+1}}\right).
\end{displaymath}
Thus, up-dating the estimator when new observations are available is easy and fast.
\\
\\
We can mention here that several variants or generalizations of the Wolverton and Wagner (WW) estimator have been proposed: see Yamato \cite{Yama}, Wegman and Davies \cite{WD79}, Hall and Patil \cite{HP}. They were studied from almost sure convergence point of view, or asymptotic rates of convergence under fixed regularity assumptions. We choose to focus on Wolverton and Wagner estimator but our results and discussions may be applied to these.
\\
\\
Theoretical developments concerning either classical Parzen-Rosenblatt or WW recursive kernels estimators occurred recently following different and independent roads.
\\
On the one hand, several recent works are dedicated to efficient and data-driven bandwidth selection, see Goldenshluger and Lespki \cite{GL11} and several companion papers by these authors, or Lacour {\it et al.} \cite{LMR17} who proposed a modification of the method. The original Goldenshluger and Lepski (GL) method was difficult to implement because it turned out to be numerically consuming and with calibration difficulties, see Comte and Rebafka \cite{CR}. This is why the improvement proposed in Lacour {\it et al.} \cite{LMR17} has both theoretical and practical interest.
\\
On the other hand, the increase of computer speed and of data sets sizes made fast up-dating of estimators mandatory. The theoretical developments in this context are in the field of stochastic algorithms (see e.g. Mokkadem {\it et al.} \cite{MPS}) or in view of specific applications (see Bercu {\it et al.} \cite{BCD}).
\\
\\
Bandwidths have to be chosen for WW estimators as for Parzen-Rosenblatt ones, and this choice is crucial to obtain good performances. 
This is why we propose to extend to this context general risk study as described in Tsybakov \cite{TSYBAKOV09} and the GL method as improved by Lacour {\it et al.} \cite{LMR17}. More precisely, considering for instance $h_k = k^{-\gamma}$ for a parameter $\gamma > 0$ in formula (\ref{WW}), we study adaptive selection of $\gamma$. We prove risk bounds for the Mean Integrated Squares Error (MISE) of the resulting estimator $\widehat f_{n,\widehat{\mathbf h}_n}$ where $\widehat{\mathbf h}_n = (\widehat h_1, \dots,\widehat h_n)$ and $\widehat h_k=k^{-\widetilde\gamma}$.
\\
Amiri \cite{AMIRI} proved that for $f$ with regularity 2 and an adequate choice of the bandwidth, Parzen-Rosenblatt's estimator had asymptotical smaller risk than the WW estimator. We propose an empirical finite sample study of this question, together with an interesting insight on the gain brought by higher order kernels.
\\
Now, clearly, plugging $\widetilde\gamma_n=\widetilde\gamma(X_1, \dots, X_n)$ in the estimator makes the recursivity fail. Therefore, an adequate strategy is required, either  with initial estimation of $\gamma$ on the first $n$-sample and recursive up-dating relying on this "frozen" value on the following $N$-sample,or with adequate matrix updating for $\widetilde\gamma_{n+1}?$ selection. This is what is experimented in our final section, and empirically illustrated and discussed.
\\
\\
This paper provides in Section \ref{MSE} controls of the MSE and of the MISE of the estimator $\widehat f_{n,\mathbf h_n}$ under general regularity conditions on $f$. Then, in Section \ref{GL}, the well-known Goldenshluger-Lepski's bandwidth selection method for Parzen-Rosenblatt's estimator is extended to Wolverton-Wagner's estimator. Lastly, an estimator in the spirit of Lacour {\it et al.} \cite{LMR17} is studied from both theoretical and practical point of view in Section \ref{LMR}.
In particular, a recursive global strategy is proposed. Concluding remarks are given in Section \ref{conclusion}. Proofs are relegated in Section \ref{Proofs}.
\\
\\
\textbf{Notation.}
\begin{enumerate}
 \item Consider $\alpha\in (0,1)$. The space of $\alpha$-H\"older continuous functions from ${\mathbb T}$ into $\mathbb R$ is denoted by $C^{\alpha}(\mathbb T)$ and equipped with the $\alpha$-H\"older semi-norm $\|.\|_{\alpha}$ defined by
 \begin{displaymath}
 \|\varphi\|_{\alpha} :=
 \sup_{x,y\in\mathbb T : x\not= y}
 \frac{|\varphi(y) -\varphi(x)|}{|y - x|^{\alpha}}
 \textrm{ $;$ }
 \forall\varphi\in C^{\alpha}(\mathbb T).
 \end{displaymath}
 \item Let $\beta , L> 0$, and ${\mathbb T}$ an interval of ${\mathbb R}$. The H\"older class $\Sigma(\beta, L)$ on ${\mathbb T}$ is the set of functions $\varphi~: {\mathbb T} \rightarrow {\mathbb R}$ such that $\varphi^{(\ell)}$, where  $\ell :=\lfloor\beta\rfloor$, is the greatest integer less than or equal to $\beta$, exists and satisfies $\|\varphi\|_{\beta -\ell}\leqslant L$.
 \item Let $\beta , L> 0$. The Nikol'ski class ${\mathcal H}(\beta, L)$ is the set of function $\varphi: {\mathbb R}\rightarrow {\mathbb R}$ such that $\varphi^{(\ell)}$ exists and satisfies 
 $$\left[\int_{-\infty}^{\infty}\left(\varphi^{(\ell)}(x+t)-\varphi^{(\ell)}(x)\right)^2dx \right]^{1/2} \leqslant L |t|^{\beta-\ell}, \quad \forall t \in {\mathbb R}.$$
 %
 \item For every square integrable function $f,g :\mathbb R\rightarrow\mathbb R$, $\|f\|_2^2=\int_{-\infty}^{+\infty}f^2(x)dx$, $\langle f, g\rangle= \int_{-\infty}^{+\infty} f(x)g(x)dx$, and 
 \begin{displaymath}
 (f\ast g)(x) :=
 \int_{-\infty}^{\infty}
 f(x - y)g(y)dy
 \textrm{ $;$ }
 x\in\mathbb R.
 \end{displaymath}
 \item $K_{\varepsilon} := (1/\varepsilon) K(\cdot /\varepsilon)$ for every $\varepsilon > 0$.
\end{enumerate}
The definitions of $\Sigma(\beta, L)$ and ${\mathcal H}(\beta, L)$ can be found in Tsybakov~(2009, Chapter 1). 

%
\section{Bounds on the MSE and the MISE of Wolverton-Wagner's estimator}\label{MSE}
Consider $\beta > 0$ and $l :=\lfloor\beta\rfloor$. Throughout this section, the map $K$ fulfills the following assumption.
%


%
\begin{assumption}\label{conditions_Kernel}
The map $y\in\mathbb R\mapsto y^iK(y)$ is integrable for every $i\in \{0, 1, \dots, l\}$,
\begin{displaymath}
\int_{-\infty}^{\infty}K(y)dy = 1, \quad \int_{-\infty}^{\infty}K^2(y)dy<+\infty, \quad \int_{-\infty}^{\infty}|z|^{\beta}|K(z)|dz:=C_\beta(K)<+\infty, 
\end{displaymath}
\begin{displaymath} \quad\mbox{ and }\quad
\int_{-\infty}^{\infty} y^iK(y)dy = 0,\;
\forall
i\in \{1, \dots, l\}. 
\end{displaymath}
\end{assumption}
\noindent
Let us establish a control of the MSE of Wolverton-Wagner's estimator under the following condition on $f$.
%


%
\begin{assumption}\label{Holder_condition}
The map $f$ belongs to the H\"older ball $\Sigma(\beta, L)$.
\end{assumption}
%


%
\begin{proposition}\label{control_MSE_f}
Under Assumptions \ref{conditions_Kernel} and \ref{Holder_condition}, there exists a constant $c > 0$, such that
for all $x \in {\mathbb T}$
\begin{displaymath}
\mathbb E(|\widehat f_{n,\mathbf h_n}(x) - f(x)|^2)
\leqslant
\frac{c}{n^2}\left(\left|\sum_{k = 1}^{n}
\frac{h_{k}^{\beta}}{l!}\right|^2 +\sum_{k = 1}^{n}\frac{1}{h_k}\right), 
\end{displaymath}
where $c$ depends on $\beta, L, \|K\|_2$ and $C_\beta(K)$ (but not on $n$ and $h_1,\dots,h_n$). 
\end{proposition}
\noindent
Now, let us establish a control of the MISE of Wolverton-Wagner's estimator under Nikolski's condition on $f$.
%


%
\begin{assumption}\label{Nikolski_condition}
The map $f$ belongs to the Nikol'ski ball ${\mathcal H}(\beta, L)$.
%
\end{assumption}
%


%
\begin{proposition}\label{control_MISE_f}
Under Assumptions \ref{conditions_Kernel} and \ref{Nikolski_condition}, there exists a constant $c > 0$, such that 
\begin{displaymath}
\int_{-\infty}^{\infty}
\mathbb E(|\widehat f_{n,\mathbf h_n}(x) - f(x)|^2)dx
\leqslant
\frac{c}{n^2}\left(\left|\sum_{k = 1}^{n}\frac{h_{k}^{\beta}}{(l - 1)!}\right|^2 +
\sum_{k = 1}^{n}\frac{1}{h_k}\right),
\end{displaymath}
where $c$ depends on $\beta, L, \|K\|_2$ and $C_\beta(K)$ (but not on $n$ and $h_1,\dots,h_n$). 
\end{proposition}

\noindent {\bf Remark.} Assumptions \ref{conditions_Kernel}, \ref{Holder_condition} and \ref{Nikolski_condition} are standard for density estimation, see Tsybakov~\cite{TSYBAKOV09}. Moreover, if we set $h_k=h$, we recover the results stated in Section 1.2.1 for Proposition \ref{control_MSE_f} and in Theorem 1.3 for Proposition \ref{control_MISE_f} in Tsybakov~(\cite{TSYBAKOV09} (that is a squared bias term of order $h^{2\beta}$ and a variance term of order $1/(nh)$).
\\

\noindent
The estimator is consistent if the risk tends to zero when $n$ grows to infinity, that is if $(1/n^2)
\left|\sum_{k=1}^{n}h_k^\beta\right|^2$ and $(1/n^2)
\sum_{k=1}^{n} \frac 1{h_k}$ tend to 0 when $n$ tends to infinity.

%
\noindent Let us consider
\begin{displaymath}
h_k = k^{-\gamma}
\textrm{ $;$ }
k\in \{1, \dots, n\}
\end{displaymath}
and, for this collection of bandwidths, set 
\begin{equation}\label{BV}
\mathbb B_n(\gamma):=\frac{1}{n^2}
\left|\sum_{k=1}^{n}h_k^\beta\right|^2
\quad\mbox{ and }\quad
\mathbb V_n(\gamma):=\frac{1}{n^2}
\sum_{k=1}^{n} \frac 1{h_k} 
\end{equation}
with $\gamma\in (0,1)$ (otherwise $\mathbb B_n(\gamma)$ or $\mathbb V_n(\gamma)$ cannot tend to zero). Then
\begin{equation}\label{bgamma}
\mathbb B_n(\gamma) =
O\left(\frac{1}{n^2}\right) \mbox{ if } \gamma\beta > 1
\quad\mbox{ and }\quad
\mathbb B_n(\gamma) =
O\left(\frac{1}{n^{2\gamma\beta}}\right) \mbox{ if } \gamma\beta < 1
\end{equation}
with the intermediate case
\begin{displaymath}
\mathbb B_n(\gamma) = O\left(\frac{\log(n)}{n^2}\right) \mbox{ if } \gamma\beta=1.
\end{displaymath}
 Indeed, if $\gamma\beta < 1$, then
\begin{displaymath}
\mathbb B_n(\gamma) =
\left|\frac{1}{n}\sum_{k = 1}^{n}k^{-\gamma\beta}\right|^2 =
n^{-2\gamma\beta}\left|\frac{1}{n}\sum_{k = 1}^{n}\left(\frac{k}{n}\right)^{-\gamma\beta}\right|^2\\
\sim\frac{n^{-2\gamma\beta}}{(1 -\gamma\beta)^2}.
\end{displaymath}
 On the other hand,
\begin{equation}\label{vgamma}
\mathbb V_n(\gamma) = O(n^{\gamma-1}).
\end{equation}
As a consequence, we have the following result:
%


%
\begin{corollary}\label{optimal_rate}
Under Assumptions \ref{conditions_Kernel} and \ref{Nikolski_condition}, choosing
\begin{displaymath}
h_k = k^{-\gamma}
\textrm{ $;$ }
k\in\{0, 1, \dots, n\}, \quad \mbox{ with } \quad \gamma=\frac 1{2\beta +1}
\end{displaymath}
yields the rate
\begin{displaymath}
\sup_{\{f\in {\mathcal H}(\beta,L), f\geqslant 0, \int f=1\}}  \int_{-\infty}^{\infty}
\mathbb E(|\widehat f_{n,\mathbf h_n}(x) - f(x)|^2)dx
\leqslant cn^{-\frac{2\beta}{2\beta +1}},
\end{displaymath}
where $c=c(\beta,L,K)$ is a positive constant depending on $L,\beta$ and the kernel $K$, but not on $n$. 
\end{corollary}
\noindent
Clearly, this is the optimal rate in the minimax sense, see Goldenshluger and Lepski \cite{GL11} and the references therein. The bounds are uniform on the set of densities belonging to the ball ${\mathcal H}( \beta,L)$.
%


%
\begin{proof}
Consider
\begin{displaymath}
\varphi_n(\gamma) :=
n^{-2\gamma\beta} +
n^{\gamma - 1}.
\end{displaymath}
Then,
\begin{displaymath}
\frac{\partial\varphi_n(\gamma)}{\partial\gamma} =
\log(n)(-2\beta e^{-2\gamma\beta\log(n)} + e^{(\gamma - 1)\log(n)}).
\end{displaymath}
Moreover, $\partial_{\gamma}\varphi_n(\gamma) = 0$ if and only if,
\begin{displaymath}
\gamma =
\frac{1}{2\beta + 1} +\frac{\log(2\beta)}{\log(n)(1 + 2\beta)}
\sim\frac{1}{2\beta + 1}.
\end{displaymath}
Therefore, $\gamma = 1/(2\beta + 1)$ makes the upper bound on the risk minimal.
\end{proof}
%


%
\section{Goldenshluger-Lepski's method for Wolverton-Wagner's estimator}\label{GL}
This section provides an extension of the well-known Goldenshluger-Lepski's bandwidth selection method for Parzen-Rosenblatt's estimator to Wolverton-Wagner's estimator.
\\
\\
Throughout this section, assume that
\begin{displaymath}
h_k = h_k(\gamma)
\textrm{ $;$ }
\forall k\in \{1, \dots, n\},
\end{displaymath}
where $\gamma\in [0,1]$ and the maps $h_1(.),\dots,h_n(.)$ from $[0,1]$ into $(0,\infty)$ fulfill the following assumption.
%


%
\begin{assumption}\label{assumption_maps_h}
For every $\gamma'\in [0,1]$,
\begin{displaymath}
0 < h_n(\gamma') <\dots < h_1(\gamma').
\end{displaymath}
Moreover, $h_n(.)$ is decreasing and one to one from $[0,1]$ into $(0,1]$.
\end{assumption}
\noindent
For instance, one can take as above $h_k(\gamma') := k^{-\gamma'}$ for every $k\in\{1, \dots, n\}$ and $\gamma'\in [0,1]$.
\\
\\
Consider
\begin{displaymath}
\mathbf h_n(\gamma) := (h_1(\gamma),\dots,h_n(\gamma))
\end{displaymath}
and the set $\Gamma_n :=\{\gamma_1,\dots,\gamma_{N(n)}\}\subset [0,1]$, where $N(n)\in \{1, \dots, n\}$ and
\begin{displaymath}
0 <\gamma_1 <\dots <\gamma_{N(n)}\leqslant h_{n}^{-1}(1/n).
\end{displaymath}
Consider also
\begin{displaymath}
\widehat f_{n,\gamma,\gamma'}(x) :=
\frac{1}{n}\sum_{k = 1}^{n}
(K_{h_k(\gamma')}\ast K_{h_k(\gamma)})(X_k - x),
\end{displaymath}
where $\gamma'\in [0,1]$.
\\
\\
A way to extend the Goldenshluger-Lepski bandwidth selection method to Wolverton-Wagner's estimator is to solve the minimization problem
\begin{equation}\label{Goldenshluger_Lepski_maximization_problem}
\min_{\gamma\in\Gamma_n}
(A_n(\gamma) + V_n(\gamma)),
\end{equation}
where
\begin{displaymath}
A_n(\gamma) :=\sup_{\gamma'\in\Gamma_n}(\|
\widehat f_{n,\mathbf h_n(\gamma')} -
\widehat f_{n,\gamma,\gamma'}\|_{2}^{2} - V_n(\gamma'))_+, \quad 
V_n(\gamma') :=\upsilon\frac{\|K\|_2^2\|K\|_1^2}{n \, \mathfrak h_n(\gamma')}
\end{displaymath}
with $\upsilon > 0$ not depending on $n$ and
\begin{displaymath}
\frac 1{\mathfrak h_n(\gamma')} :=
\frac{1}{n}\sum_{k = 1}^{n}\frac{1}{h_k(\gamma')}.
\end{displaymath}
In the sequel, the map $\mathfrak h_n(.)$ fulfills the following assumption.
%


%
\begin{assumption}\label{assumption_Gamma_n}
For every $c > 0$ and $r\in\{1/2,1\}$,
\begin{displaymath}
\sup_{n\in\mathbb N^*}
\sum_{\gamma'\in\Gamma_n}
\exp(-c/\mathfrak h_n(\gamma')^r)
<\infty.
\end{displaymath}
\end{assumption}
\noindent
\textbf{Example.} Consider
\begin{displaymath}
h_k(\gamma') =
k^{-\gamma'}
\textrm{ $;$ }
\forall k\in \{1, \dots, n\}
\textrm{, }
\forall\gamma'\in [0,1]
\end{displaymath}
and
\begin{equation}\label{GammaEx}
\Gamma_n =\left\{\left(
\frac{i}{[\log(n)]}\right)^{1/2}\textrm{ $;$ }i\in \{1, \dots, [\log(n)]\} \right\},
\end{equation}
where $[x]$ denotes the interger part of $x$. 
For every $\gamma'\in\Gamma_n$,
\begin{eqnarray*}
\frac 1{\mathfrak h_n(\gamma')} & = &
 \frac{1}{n}\sum_{k = 1}^{n}\frac{1}{h_k(\gamma')}
 = n^{\gamma' - 1}\sum_{k = 1}^{n}\left(\frac{k}{n}\right)^{\gamma'}\\
 & \geqslant & n^{\gamma' - 2}\sum_{k = 1}^{n}k
 \geqslant\frac{n^{\gamma'}}{2}
 \geqslant
 \frac{1}{2}\exp(\log(n)^{1/2}).
\end{eqnarray*}
Then, for any $c > 0$ and $r\in\{1/2,1\}$,
\begin{displaymath}
\sup_{n\in\mathbb N^*}
\sum_{\gamma'\in\Gamma_n}
\exp(-c/\mathfrak h_n(\gamma')^r)
\leqslant
\sup_{n\in\mathbb N^*}
\log(n)\exp\left(-\frac{c}{2^r}\exp(r\log(n)^{1/2})\right)
<\infty.
\end{displaymath}
%


%
\begin{proposition}\label{Goldenshluger_Lepski}
Under Assumptions \ref{assumption_maps_h} and \ref{assumption_Gamma_n}, if $f$ is bounded and $\widehat\gamma_n$ is a solution of the minimization problem (\ref{Goldenshluger_Lepski_maximization_problem}), then there exists a constant $\upsilon_0$ such that, for $\upsilon\geqslant \upsilon_0$, 
\begin{displaymath}
\mathbb E(\|\widehat f_{n,\mathbf h_n(\widehat\gamma_n)} - f\|_{2}^{2})
\leqslant
c_0
\inf_{\gamma\in\Gamma_n}
\left(
V_n(\gamma) +
\frac{1}{n^2}\left|\sum_{k = 1}^{n}\|f - K_{h_k(\gamma)}\ast f\|_2\right|^2 \right)+
\frac{c_1}{n},
\end{displaymath}
where $c_0$ is a numerical constant ($c_0=18$ suits) and $c_1$ is a constant depending on $K$ and $\|f\|_\infty$ (but not on $n$). \\
If in addition Assumptions \ref{conditions_Kernel} and \ref{Nikolski_condition} hold, then
\begin{equation}\label{GLresult}
\mathbb E(\|\widehat f_{n,\mathbf h_n(\widehat\gamma_n)} - f\|_{2}^{2})
\leqslant
c\left\{
\inf_{\gamma\in\Gamma_n}
\left({\mathbb B}_n(\gamma)+ {\mathbb V}_n(\gamma)\right) +\frac 1n\right\}
\end{equation}
where ${\mathbb B}_n(\gamma)$ and ${\mathbb V}_n(\gamma)$ are defined in (\ref{BV}), (\ref{bgamma}) and (\ref{vgamma}).
\end{proposition}

\noindent Note that the proof leads to the value $\kappa_0=24$, which would be too large in practice. \\

\noindent
\textbf{Remark.} By Corollary \ref{optimal_rate}, the infimum in bound (\ref{GLresult}) has the order of the optimal rate, and is reached automatically by the data driven estimator. This result is more precise than the heuristics associated with cross-validation. 
\\
We mentioned previously that the optimal theoretical choice for $\gamma$ under Assumptions \ref{conditions_Kernel} and \ref{Nikolski_condition} is $\gamma=1/(2\beta+1)$. Here, the selected $\gamma$ should be at nearest of this value, e.g. if $\Gamma_n$ is as in (\ref{GammaEx}), distant from less than $1/\sqrt{\log(n)}$ of the good choice. We may therefore consider that $\widehat\gamma_n$ provides an estimate of $1/(2\beta+1)$ and thus an estimate of the regularity $\beta$ of $f$ (at least for huge values of $n$).
%


%
\section{The Lacour-Massart-Rivoirard  (LMR) estimator}\label{LMR}
\subsection{Estimator and main result}
The Goldenshluger-Lepski method has been acknowledged as being difficult to implement, due to the square grid in $\gamma,\gamma'$ required to compute intermediate versions of the criterion and to the lack of intuition in the choice of the constant $\upsilon$ which should be calibrated from preliminary simulation experiments. This is the reason why Lacour {\it et al.} \cite{LMR17} investigated and proposed a simplified criterion relying on deviation inequalities for $U$-statistics due to Houdr\'e and Reynaud-Bouret \cite{HRB03}. This inequality applies in our more complicated context and Lacour-Massart-Rivoirard's result can be extended here as follows. 
\\
\\
Let us recall that $K_{\varepsilon}(\cdot) := (1/\varepsilon) K(\cdot /\varepsilon)$ for every $\varepsilon > 0$ and set
\begin{displaymath}
f_{n,\gamma}(x) :=
\mathbb E(\widehat f_{\mathbf h_n(\gamma)}(x)) =
\frac{1}{n}
\sum_{k = 1}^{n}
(K_{h_k(\gamma)}\ast f)(x).
\end{displaymath}
Let $\gamma_{\max}$ be the maximal proposal in $\Gamma_n$ and consider
\begin{displaymath}
\textrm{Crit}(\gamma) :=
\|\widehat f_{n,\mathbf h_n(\gamma)}-\widehat f_{n,\mathbf h_n(\gamma_{\max})}\|_{2}^{2} +\textrm{pen}(\gamma)
\end{displaymath}
with
\begin{displaymath}
\textrm{pen}(\gamma) :=
\frac{2}{n^2}
\sum_{k = 1}^{n}
\langle K_{h_k(\gamma_{\max})},K_{h_k(\gamma)}\rangle_2.
\end{displaymath}
Then, we define
\begin{displaymath}
\widetilde\gamma_n\in
\arg\min_{\gamma\in\Gamma_n}
\textrm{Crit}(\gamma).
\end{displaymath}
In the sequel, $K$, $f$ and $\mathbf h_n$ fulfill the following assumption.
%


%
\begin{assumption}\label{conditions_LMR}
The kernel $K$ is symmetric, $K(0) > 0$,
\begin{displaymath}
\int_{-\infty}^{\infty}K(y)dy = 1
\mbox{, }\quad
\frac{\|K\|_{\infty}\|K\|_1}{nh_n(\gamma_{\max})}\leqslant 1
\end{displaymath}
and $\|f\|_{\infty} <\infty$.
\end{assumption}
%


%
\begin{proposition}\label{LMR_method}
Consider $\lambda\in [1,\infty[$ and $\varepsilon\in (0,1)$. Under Assumption \ref{conditions_LMR}, there exists three deterministic constants $c_1,c_2,c_3 > 0$, not depending on $n$, $\lambda$ and $\gamma$, such that with probability larger than $1 - c_1|\Gamma_n|e^{-\lambda}$,
\begin{eqnarray*}
 \| \widehat f_{n,\mathbf h_n(\widetilde\gamma_n)}-f\|_{2}^{2} &\leqslant &
 (1+\varepsilon)\min_{\gamma\in \Gamma_n}\|\widehat f_{n,\mathbf h_n(\gamma)}-f\|_{2}^{2}\\
 & & + \frac{c_2}{\varepsilon}\| f_{n,\gamma_{{\rm max}}} - f\|_{2}^{2}
 + \frac{c_3}{\varepsilon} \left(\frac{\lambda^2}n  + \frac{\lambda^3}{n^2h_n(\gamma_{\max})}\right).
\end{eqnarray*}
\end{proposition}
\noindent
\textbf{Remark.} The term $\| f_{n,\gamma_{{\rm max}}} - f\|_{2}^{2}$ is negligible because it is a pure bias term for smallest bandwidth (e.g., under Assumption \ref{Nikolski_condition}, it has order $n^{-2\beta\gamma_{{\rm max}}}$, see (\ref{bgamma}), and thus $o(1/n)$ if $\gamma_{{\rm max}}$ is near of 1 and $\beta>1/2$). The terms following are of order $O(1/n)$ and are always negligible compared to nonparametric rates in our setting. Therefore, the bound given in Proposition \ref{LMR_method} says that the MISE of the adaptive estimator has the order of the best estimator of the collection, up to a multplicative factor larger than 1. This is the method we implement in the next section: it is faster than GL method and with no constant to calibrate in the penalty.
%


%
\subsection{Simulation experiments}\label{simu}
We consider basic densities with different types and orders of regularity: \begin{itemize}
\item $X\rightsquigarrow\mathcal N(0,1)$, density $f_1$,
\item a mixed gaussian $X\rightsquigarrow 0.5\mathcal N(-2,1) + 0.5\mathcal N(2,1)$, density $f_{m,1}$,
\item $X\rightsquigarrow\beta(3,3)$, density $f_2$,
\item a mixed beta $X\rightsquigarrow 0.5(\beta(3,3) - 1)+ 0.5 \beta(3,3)$, density $f_{m,2}$,
\item $X\rightsquigarrow\gamma(5,5)/10$, density $f_3$,
\item a mixed gamma $X\rightsquigarrow 0.4.\gamma(2,1/3)+0.6 \gamma(7,6)/10$, density $f_{m,3}$,
\item $X\rightsquigarrow f_4$ with $f_4(x)=e^{-|x|}$, a Laplace density.
\end{itemize}
The densities $f_1$ and $f_{m,1}$ have infinite regularity, $f_2$ and $f_{m,2}$ should rather have regularity of order less than 2, $f_3$ and $f_{m,3}$ less than 4, and $f_4$ less than 1. This choice should allow to study the influence of the order of the kernel.
\\
\begin{table}
\begin{tabular}{cc|cccc|cccc|c}
&& \multicolumn{4}{|c|}{LMR for WW} & \multicolumn{4}{|c|}{ Original LMR } &  \\ 
&$n=$ & $K_1$ & $K_3$ & $K_5$ & $K_7$ & $K_1$ & $K_3$ & $K_5$ & $K_7$ & {\tt ks} \\
\hline
&&&&&&&&&&\\
$f_1$ & $250$ & 0.442 & 0.318 &  0.285  &  0.256 &  
							0.412 & 0.315 &  0.290  &  0.268 &   0.285\\
&  &       {\tiny (0.252)} & {\tiny (0.213)} &  {\tiny (0.193)} &  {\tiny (0.162)} &   
							{\tiny (0.241)} & {\tiny (0.214)} &  {\tiny (0.205)}  & 
							{\tiny (0.193)} &  {\tiny (0.174)} \\
& $1000$ &   0.144 &   0.091 & 0.080 &    0.075 &  0.133 & 0.088 & 0.079 & 0.076  &  0.101 \\
&&   {\tiny (0.079)}   & {\tiny (0.065)} &  {\tiny (0.061)} & {\tiny (0.059)} & {\tiny (0.076)}  & {\tiny (0.064)} & {\tiny (0.062)} &  {\tiny (0.061)}   & {\tiny (0.059)} \\
$f_{m,1}$ & $250$ &  0.400 &  0.316 &  0.287 &  0.255 & 0.387 &  0.327 &    0.291 &   0.256 &    1.115 \\
  && {\tiny (0.204)} & {\tiny (0.189)} & {\tiny (0.176)} & {\tiny (0.162)} & {\tiny (0.208)} & {\tiny (0.202)} & {\tiny (0.179)} & {\tiny (0.170)} &    {\tiny (0.150)} \\
& 1000 &   0.141 &  0.101 & 0.090 & 0.084 & 0.135 & 0.101 & 0.094 &  0.091 & 0.585 \\
&&  {\tiny (0.0623)} & {\tiny (0.051)} & {\tiny (0.049)} & {\tiny (0.046)} &{\tiny (0.062)} & {\tiny (0.053)} & {\tiny (0.051)} & {\tiny (0.050)} & {\tiny (0.076)} \\
$f_2$ & 250 &  3.586 & 2.141 & 1.840 & 1.709 & 1.865 & 1.343 & 1.221 & 1.178 & 1.272 \\
 &&  {\tiny (1.403)} &  {\tiny (1.230)} & {\tiny (1.155)} & {\tiny (1.116)} & {\tiny (1.108)} & {\tiny (0.930)} &  {\tiny (0.884)} & {\tiny (0.885)} &{\tiny (0.789)}\\
& 1000 &   1.056 &  0.646 &  0.555 & 0.515 & 0.602 & 0.429 & 0.382 & 0.372 &   0.506 \\
&&   {\tiny (0.394)} &  {\tiny (0.306)} &  {\tiny (0.283)} & {\tiny (0.270)} & {\tiny (0.312)} & {\tiny (0.270)} & {\tiny (0.250)} & {\tiny (0.235)} & {\tiny (0.282)} \\
$f_{m,2}$ & $250$ &   3.071 & 2.040 & 1.778 &  1.654 & 1.825 & 1.362 &  
1.217 &   1.157 &  8.912 \\
 &&  {\tiny (0.851)} & {\tiny (0.743)} & {\tiny (0.706)} & {\tiny (0.681)} &  {\tiny (0.655)} & {\tiny (0.584)} & {\tiny (0.605)} &  {\tiny (0.565)} & {\tiny (0.909)}\\
& $1000$ &    0.905 & 0.593 & 0.508 & 0.476 & 0.657 & 0.438 &  0.389 &     0.358 &   4.876\\
&&   {\tiny (0.246)} &  {\tiny (0.201)} &  {\tiny (0.187)} &  {\tiny (0.182)} &  {\tiny (0.257)} & {\tiny (0.188)} & {\tiny (0.159)} &  {\tiny (0.163)} & {\tiny (0.367)}\\
$f_3$ & $250$ &  0.449 & 0.358 & 0.340 &  0.326 & 0.419 &  0.356 &  0.343 & 0.327 &  0.298 \\
 &&  {\tiny (0.263)} & {\tiny (0.236)} & {\tiny (0.221)} & {\tiny (0.198)} & {\tiny (0.259)} & {\tiny (0.241)} & {\tiny (0.224)} & {\tiny (0.201)} & {\tiny (0.202)} \\
& $1000$ & 0.174 & 0.132 & 0.124 &  0.121 & 0.162 & 0.130 &  0.126 & 0.126 & 0.125 \\
&&  {\tiny (0.085)} & {\tiny (0.071)} & {\tiny (0.067)} & {\tiny (0.065)} &{\tiny (0.081)} &  {\tiny (0.071)} & {\tiny (0.071)} & {\tiny (0.076)} & {\tiny (0.065)} \\
$f_{m,3}$ & $250$ & 1.257 & 1.129 & 1.106 & 1.103 & 1.140 &  1.117 & 1.138 & 1.162 &  4.089 \\
&&   {\tiny (0.597)} & {\tiny (0.555)} & {\tiny (0.537)} & {\tiny (0.532)} & {\tiny (0.564)} & {\tiny (0.562)} & {\tiny (0.568)} & {\tiny (0.576)} & {\tiny (0.355)} \\
& $1000$ &    0.491 &  0.448 & 0.444 & 0.446 & 0.449 &  0.441 &  0.454 & 
 0.466 &  3.172 \\
&&   {\tiny (0.171)} & {\tiny (0.158)} & {\tiny (0.158)} & {\tiny (0.160)} & {\tiny (0.168)} & {\tiny (0.174)} & {\tiny (0.189)} & {\tiny (0.204)} & {\tiny (0.201)} \\
$f_4$ & $250$ &  0.683 & 0.642 & 0.642 & 0.649 & 0.663 & 0.680 & 0.706 &    0.708 &  0.519 \\
&&   {\tiny (0.353)} & {\tiny (0.318)} & {\tiny (0.301)} & {\tiny (0.294)} &  {\tiny (0.347)} & {\tiny (0.343)} & {\tiny (0.339)} & {\tiny (0.322)} & {\tiny (0.260)} \\
& $1000$ &   0.281 &   0.254 &  0.254 &  0.258 & 0.273 &  0.268 & 0.278 &  0.284 &  0.242 \\
&&  {\tiny (0.135)} & {\tiny (0.122)} & {\tiny (0.120)} & {\tiny (0.122)} & {\tiny (0.141)} & {\tiny (0.147)} & {\tiny (0.163)} & {\tiny (0.172)} & {\tiny (0.105)}\\
&&&&&&&&&&\\
\end{tabular}
\caption{100 $\times$ MISE with 100 $\times$ std in parenthesis, computed over 200 simulations.}\label{tab1}
\end{table}
\\
Denoting by $n_j(x)$ the density of a centered Gaussian random variable with variance equal to $j$, we consider the following kernels:
\begin{itemize}
\item a Gaussian kernel, $K_1(x)=e^{-x^2/2}/\sqrt{2\pi}$ which is of order 1, 
\item a Gaussian-type kernel of order 3, $K_3(x)=2n_1(x)-n_2(x)$, 
\item a Gaussian-type kernel of order 5, $K_5(x)= 3n_1(x)-3n_2(x)+ n_3(x)$, 
\item a Gaussian-type kernel of order 7, $K_7(x)=4n_1(x)-6n_2(x)+4n_3(x)-n_4(x)$.
\end{itemize}
With all these kernels, the penalty terms are computed analytically and without approximation. Indeed, for $n_{i,h}(x) = (1/h)n_i(x/h)$, it holds that
\begin{displaymath}
\langle n_{i,h_1}, n_ {j,h_2}\rangle_2= \int_{-\infty}^{\infty}
n_{i,h_1}(x) n_{j,h_2}(x)dx =
\frac{1}{\sqrt{2\pi}}\times\frac{1}{\sqrt{ih_{1}^{2} + jh_{2}^{2}}}.
\end{displaymath}
We compute the variable bandwidth estimator as described in Section \ref{LMR} and select $\widetilde\gamma_n$ in a collection of $M = 40$ equispaced values between 0 and 0.5 while the bandwidth associated with observation $i$ is $h_i(\gamma)=i^{-\gamma}$. We also compute the original estimator of Lacour {\it et al.} \cite{LMR17} with bandwidth $h$ which does not depend on the observation and is selected among $M = 40$ values in the set $\{k/M\textrm{ $;$ } k=1, \dots, M\}$. 
\\
\\
For comparison, we give the performance of the Matlab density estimator obtained from ${\tt ksdensity}$ function (denoted by {\tt ks} in Table \ref{tab1}), which entails a different bandwidth selection method and relies on a gaussian kernel.
\\
\\
We compute the integrated ${\mathbb L}^2$-risk associated with all the final estimators, evaluated at $P = 100$ equispaced points in the range $[a,b]$ of the observations, averaged over $R = 200$ repetitions:
\begin{displaymath}
\frac{1}{R}
\sum_{j = 1}^{R}
\frac{b - a}{P}
\sum_{\ell = 1}^{P}
(\widehat f^{(j)}_{\widehat h^{(j)}}(x_\ell) - f(x_\ell))^2
\mbox{, }\quad
x_\ell = a + \ell\frac{b - a}{P},
\end{displaymath}
where $\widehat f^{(j)}_{\widehat h^{(j)}}$ is the estimator computed for path $j$.
\begin{figure}
\begin{tabular}{c}
\includegraphics[width=12cm,height=5cm]{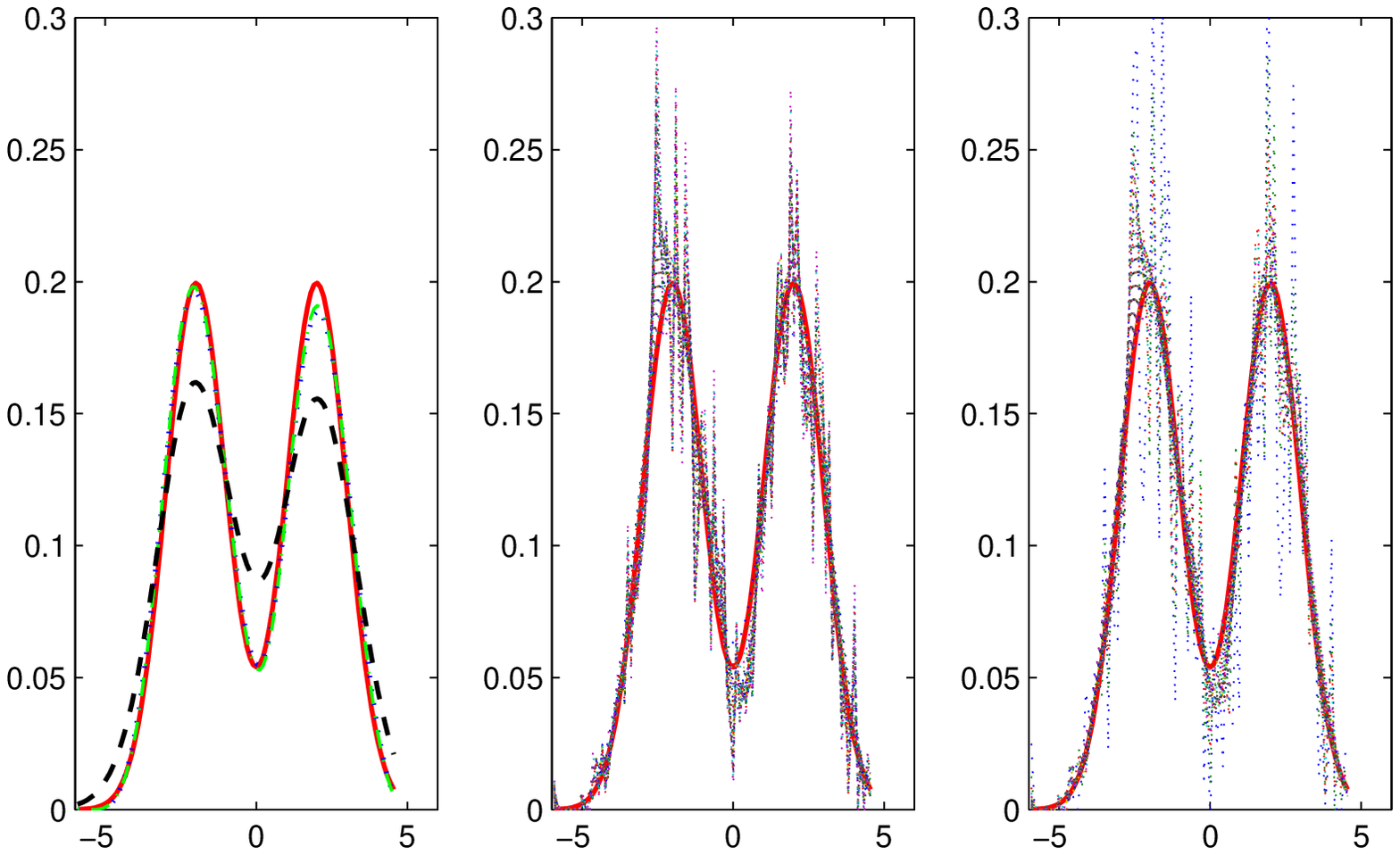}\\
\includegraphics[width=12cm,height=5cm]{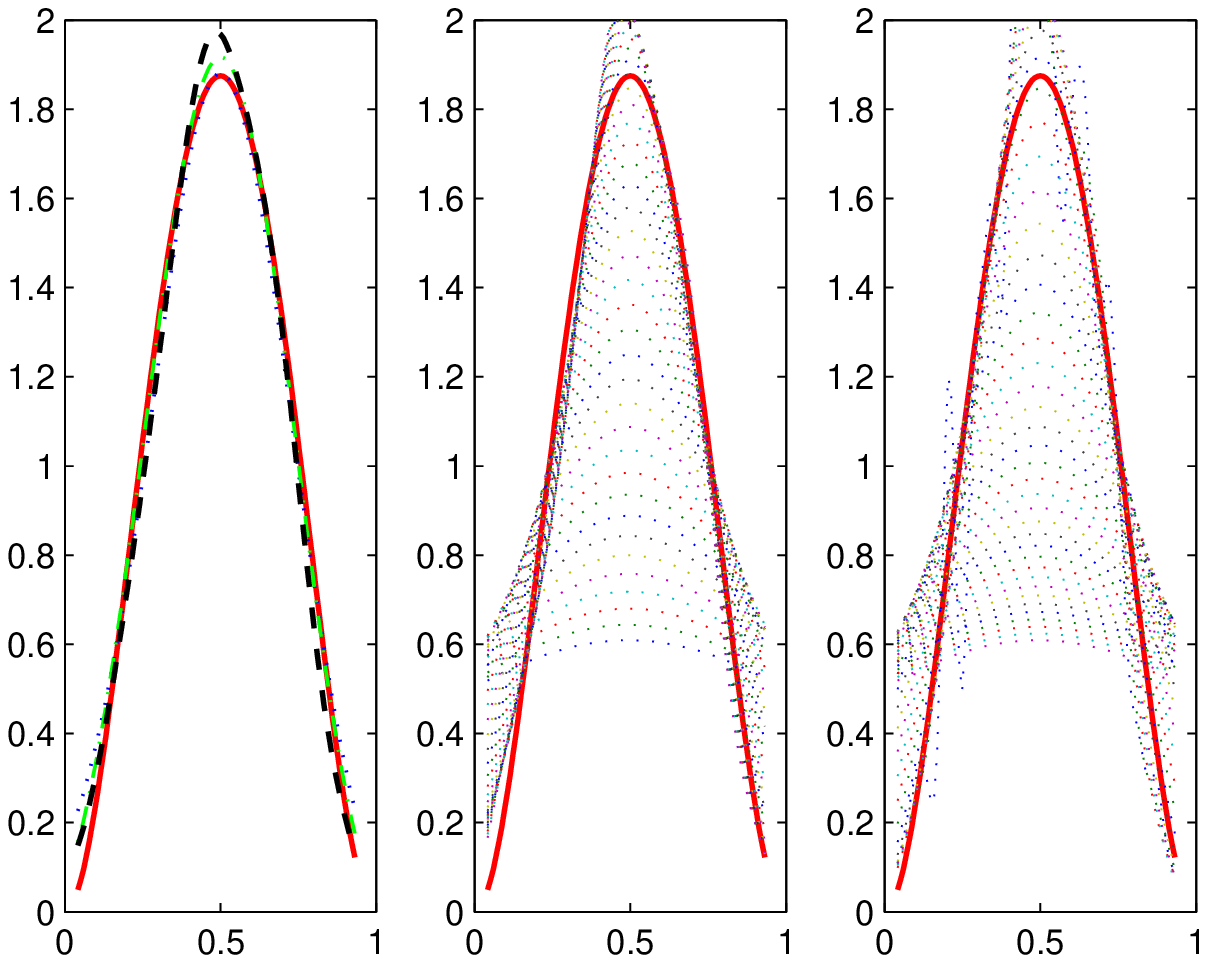}\\
\end{tabular}
\caption{Left: The three estimators (dotted blue LMR-WW, green dash-dotted LMR, black dashed {\tt ks}, the true in bold red. Middle: the 40 proposals for LMR-WW. Right: the 40 proposals for LMR. First line $n=1000$, density $f_{1,m}$,  second line $n=250$, density $f_2$. In all cases, kernel $K_7$.}\label{fig1}
\end{figure}
\begin{figure}
\includegraphics[width=12cm,height=5cm]{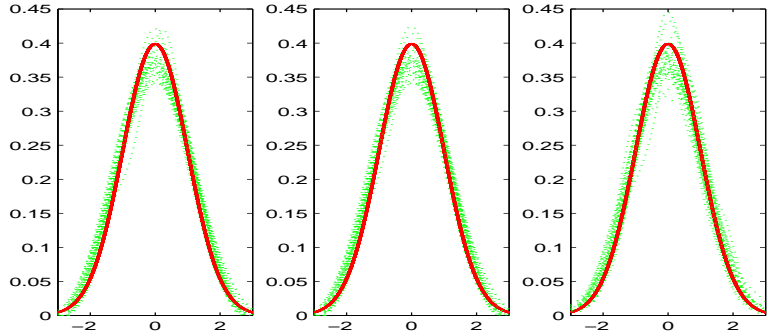}
\caption{Beams of 30 estimators in dotted green of density $f_1$ for $n=250$ and kernel $K_7$, and the true in bold red. Left: LMR-WW estimator. Middle: LMR estimator. Right: {\tt ks} estimator.}\label{fig2}
\end{figure}
Results are gathered in Table~\ref{tab1}  and deserve some comments. As expected, when increasing $n$ from 250 to 1000, the resulting MSEs decrease and seem to be more improved in LMR methods of both types than for {\tt ks} estimator. Increasing the order of the kernel systematically improves the results, except for the lowest regularity density $f_4$, which is at best with $K_3$, but it is interesting to note that taking higher order kernel is always a good strategy: if a loss occurs, it is negligible while the improvement, when it happens, is in all cases significant. Estimator {\tt ks} fails for all mixed densities $f_{m,1}$, $f_{m,2}$ and $f_{m,3}$ and provides rather bad results in these cases, for both sample sizes. For the other densities ($f_1$, $f_2$, $f_3$, $f_4$), the results obtained with kernel $K_7$ and LMR method are better than with {\tt ks}  for $f_1$ (Gaussian case), and of comparable order in all other cases. Now if we compare the LMR and WW-LMR results both with kernel $K_7$, we conclude that the WW-LMR method wins in 10 cases out of 14, but not significantly.
\\
\\
The first line of Figure \ref{fig1} illustrates in the left picture the way Matlab estimator fails for mixed densities (here the mixed Gaussian $f_{m,1}$) by probably selecting a too large bandwidth, here $n=1000$. The two LMR estimators are almost confounded. The middle and right pictures present the $M=40$ estimators among which the LMR procedure makes the selection, for the same path: we observe that the collection of proposals are rather different. The second line of Figure \ref{fig1} presents the same type of results for density $f_2$, and sample size 
$n=250$. Figure \ref{fig2} shows beams of 30 final estimators for sample size $n=250$, for the three estimators LMR-WW with $K_7$, LMR with $K_7$ and {\tt ks}, showing very similar behaviours.
\\
\\
A last remark corresponding to numerical results we do not report in detail is the following. For most densities, the value of $\gamma$ selected by the LMR strategy decreases, and the value of $h$ increases, when the order of the kernel increases. Exceptions are densities with lower regularity (the beta $f_2$, mixed beta $f_{2,m}$ and Laplace $f_4$ densities) for which the last value of selected $h$ with $K_7$ is less than the one selected with $K_5$. This illustrates the fact that, asymptotically, if $\beta$ is the regularity index of the density and $\ell$ the order of the kernel, the optimal choice is for $h$ of order $n^{-1/(2\min(\beta, \ell)+1)}$ and for $\gamma$, $1/(2\min(\beta, \ell)+1)$. This point is further investigated hereafter.
%

\subsection{Back to recursivity}

However, one may wonder how to keep these ideas compatible with recursive procedures and online updating of the kernel estimator. 

A first possibility is to consider that the adaptive bandwidth, whatever its type, can be selected on a preliminary sample and then, "frozen" to this selected value. The estimator may then be recursively updated with this frozen value, and the procedure would exploit all the observations. 

We compute the mean over 200 repetitions of the selected values  $\widetilde\gamma_n$ for $n$ observations, with increasing values of $n$, for the functions $f_1, \dots, f_4$ defined in Section \ref{simu}. We can see in Table \ref{gammaconv} that, if there is a convergence towards a value, it is very slow. Indeed, keeping in mind the value $1/(2\min(\beta, \ell)+1)$ for $\gamma$, we may expect $\widetilde \gamma_n$ to tend to $1/15=0.67$ for $f_1$, $1/5=0.2$ for $f_2$, $1/9=0.11$ for $f_3$ and a quantity less that $1/5=0.2$ for $f_4$.

\begin{table}
\begin{tabular}{c|cccc}
$n$ & 250 & 1000 & 2000 & 4000 \\ \hline 
$f_1$ &  0.017 &  0.044 &  0.054 & 0.059 \\
& \tiny{(0.023)} &   \tiny{(0.016)} & \tiny{(0.012)} & \tiny{(0.019)} \\
$f_2$ &  0.331 &   0.304   & 0.291   & 0.281\\
  &  \tiny{(0.020)}  &  \tiny{(0.014)} &   \tiny{(0.011)}  & \tiny{(0.013)}\\
$f_3$ &  0.026 &   0.067  &  0.076  &  0.084 \\
& \tiny{(0.027)} & \tiny{(0.020)} &  \tiny{(0.017)} & \tiny{(0.018)}\\
$f_4$ & 0.070 &   0.120 &   0.136  &  0.146 \\
&    \tiny{(0.040)} &   \tiny{(0.039)} &  \tiny{(0.039)} &  \tiny{(0.043)}\\
\end{tabular}
\caption{Mean (and std) of $\widetilde \gamma_n$ for different values of $n$ and functions $f_1, \dots, f_4$, 200 repetitions}\label{gammaconv} 
\end{table}

We provide in Table \ref{frozen} the results obtained for densities $f_1, \dots, f_4$ and sample size 1000 splitted in two parts: $n_0=500$ observations used  for the selection of $\widetilde\gamma_{n_0}$ and $n_1=500$ updates of the resulting estimator with the "frozen" value  $\widetilde \gamma_{n_0}$. We compute the MISE for the first step and final estimator, which relies on $n=n_0+n_1$ observations. The results are given in Table \ref{frozen}. We can see an improvement when going from $n_0$ to $n$ observations, but the results are deteriorated compared with what is obtained with $n=n_0+n_1=1000$ in Table \ref{tab1}.

\begin{table}
\begin{tabular}{cc|cc|cc|cc}
 \multicolumn{2}{c|}{Function $f_1$} & \multicolumn{2}{c|}{Function $f_2$} & 
\multicolumn{2}{c|}{Function $f_3$}  & \multicolumn{2}{c}{Function $f_4$} \\ 
 WW &  Update & WW &  Update& WW & Update& WW & Update \\
&&&&&&& \\
0.146  &    0.112 &   0.879   &  0.414   & 0.188  &  0.131   & 0.393    & 
0.303\\
\tiny{(0.105)}   &  \tiny{(0.082)}   & \tiny{(0.502)}  &  \tiny{(0.214)}  &  \tiny{(0.115)}  &  \tiny{(0.076)} &  \tiny{(0.186)}  &   \tiny{(0.146)} \\
\end{tabular}
\medskip
\caption{Comparison of 100 $\times$ MISE (with 100$\times$ std in parenthesis)
for the estimators: adaptive LMR-WW, and updated recrusively WW with "frozen" $\gamma$, first sample $n_0=500$, update sample with $n_1=500$ additional observations}\label{frozen}
\end{table}

This is why our idea is to exploit the recursive formula to select the $\gamma$-parameter at each step. The price to pay is to store a matrix instead of storing a vector, but the matrix size is fixed. Thus, for one given sample, the procedure is fast and the storage size under control. Precisely, adding an observation leads to update the matrix ${\mathbf F}_n=\left(\widehat f_{n, {\mathbf h}
_n(\gamma^{(j)})}(x_k)\right)_{1\leqslant j\leqslant M, 1\leqslant k\leqslant K}$ for $\Gamma_M=
\{ \gamma^{(1)}, \dots, \gamma^{(M)}\}$ the collection of proposed values of $\/gamma$ and $x_k, k=1, \dots, K$ the set points at which 
the function is estimated. This is the collection of values used to select $\widetilde\gamma_n$.
Then we have $${\mathbf F}_{n+1}=\frac n{n+1}{\mathbf F}_n + \frac 1{n+1}
\left(\frac 1{h_{n+1}(\gamma^{(j)})} K\left(\frac{X_{n+1}-x_k}{h_{n+1}(\gamma^{(j)})}\right)\right)_{1\leqslant j\leqslant M, 1\leqslant k\leqslant K},$$
and we can select $\widetilde \gamma_{n+1}$. The resulting estimator is already computed since it is one of the collection. At each stage, the $M\times K$ matrix is stored to be updated, and the selection procedure has entries the matrix, the value of $n$ and the domain of observations, and gives as output the selected value of $\gamma\in \Gamma_M$ and the corresponding functional estimator,  which corresponds to a line of the matrix. We show in Figure \ref{trajgamma} the 950 selected values of $\widetilde\gamma_n$ for $n=50$ to $n=1000$, for one path of a random sample with density $f_1$, \dots, $f_4$ on the left and $f_{m,1}, f_{m,2}, f_{m,3}$ on the right plot. We consider $M=50$ and $K=100$. We observe on this example a stabilization of the selected value when $n_1+n_0$ gets near of 1000 observations. 
  
\begin{figure}
 \begin{tabular}{cc}
\hspace{-1cm}\includegraphics[width=7cm,height=5cm]{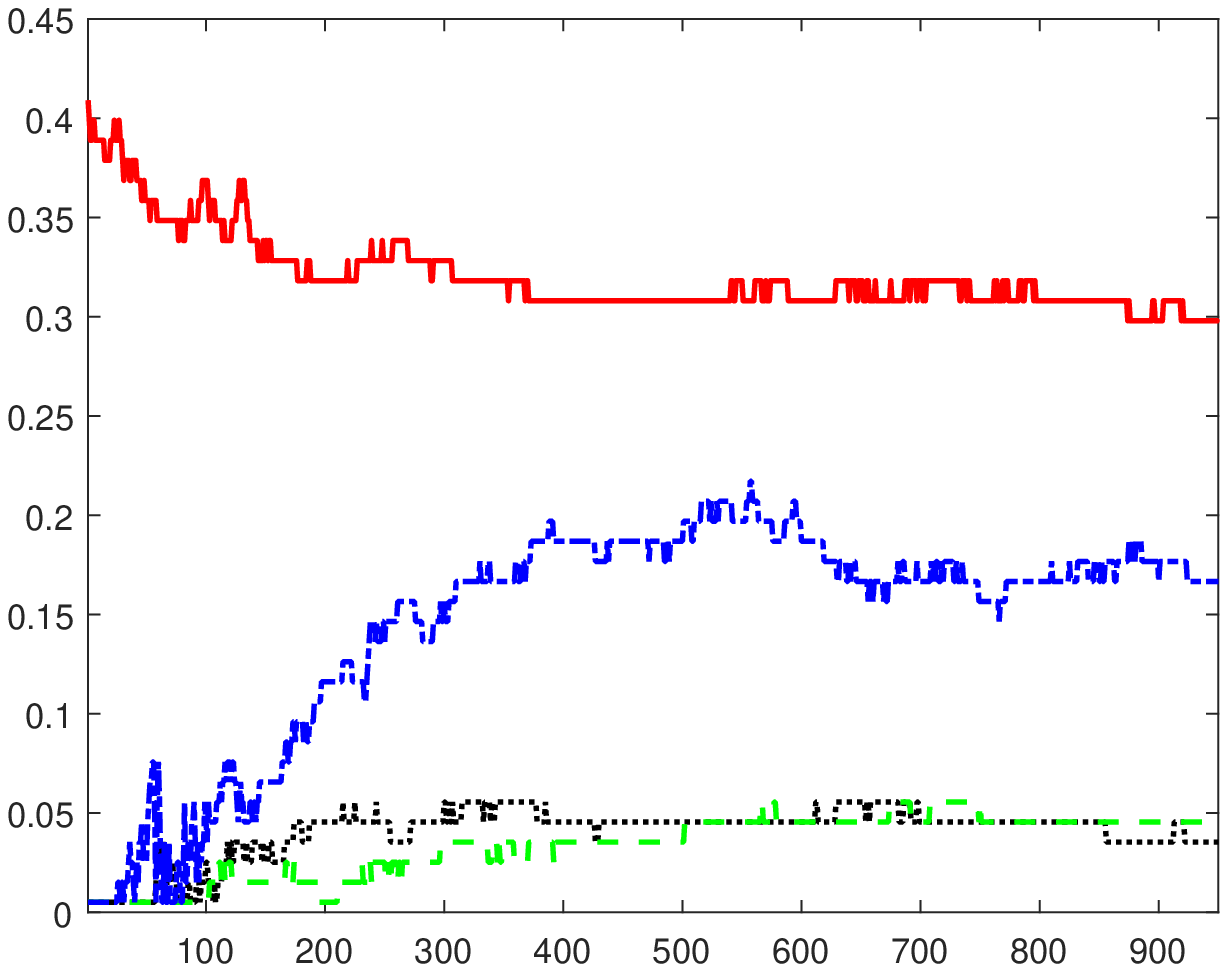}&
\includegraphics[width=7cm,height=5cm]{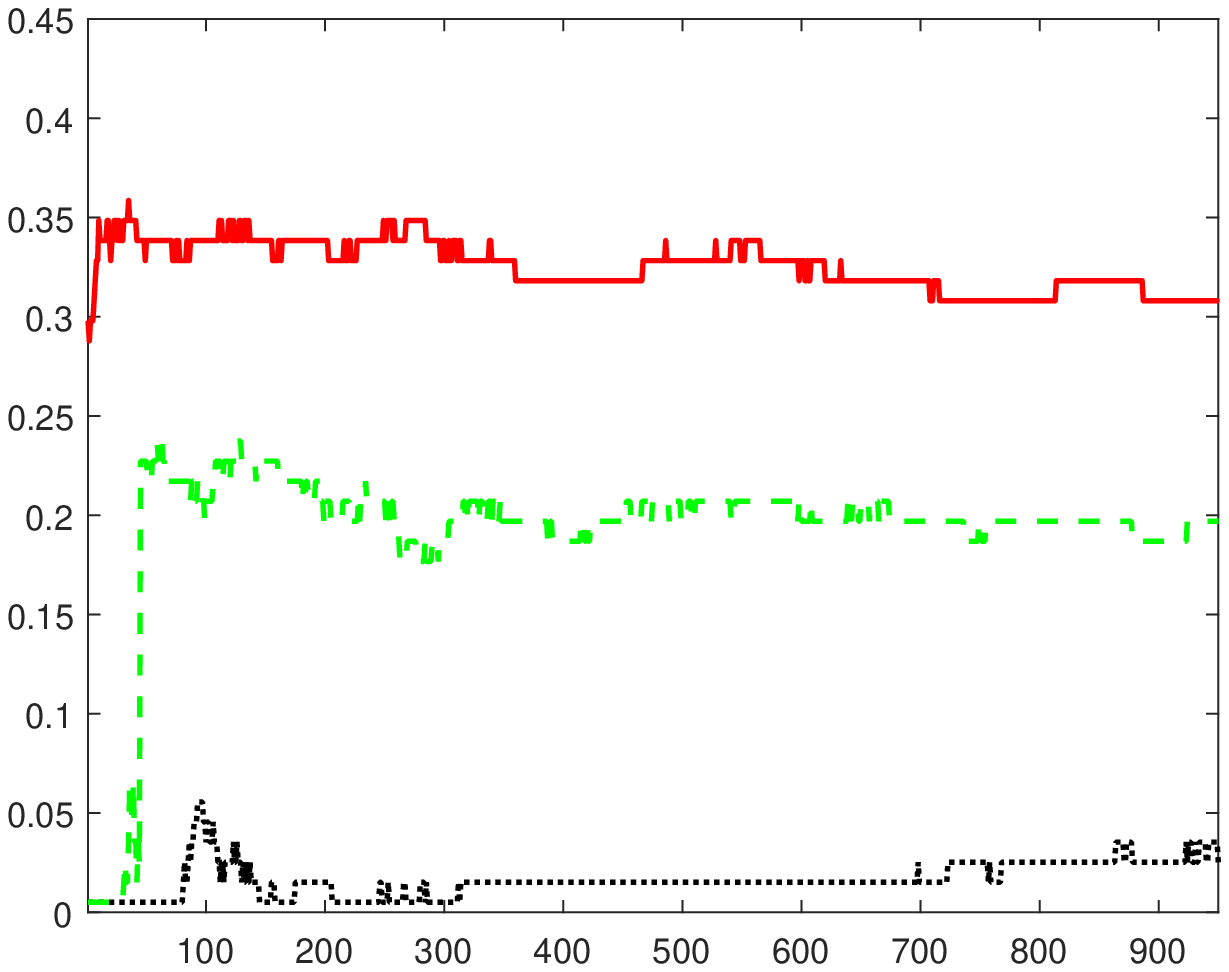}\\
\end{tabular}
\caption{Selected $\widetilde \gamma_n$ from $n=50$ to $n=1000$ for densities  $f_1$ (black dotted), $f_2$ (red line), $f_3$ (green dashed) and $f_4$ (blue dash-dotted) for the left picture, $f_{m,1}$ (black dotted), $f_{m,2}$ (red line), $f_{m,3}$ (green dashed) on the right picture.}\label{trajgamma}
\end{figure}

%
%
%
%
%
%
%
%
%
%
%
%
%
%
%
%
%
\section{Concluding remarks}\label{conclusion}
Our study illustrates that bandwidth selection is an important step for kernel functional estimation, and recent methods are really powerful whatever the type of density to recover.
\\
Our simulations show also that, even if it implies non necessarily nonnegative kernels and thus density estimators, increasing the order of the kernel improves the estimation  both in the theory and in practice. Also, we proved that variable bandwidth for WW-type estimators can reach excellent rates, again both in theory and in practice, provided that adaptive choice of this variable bandwidth is performed. The orders of practical MISEs show that this WW-strategy provides results of the same order as the more classical bandwidth methods. 
\\
Lastly, we illustrate that recursivity formula can be used for fast online updating of the whole collection of estimators and the selection of the best one in the sense of our criterion.  
\\

\section{Proofs}\label{Proofs}
%


%
\subsection{Proof of Proposition \ref{control_MSE_f}}
First, by the bias-variance decomposition,
\begin{equation}\label{bias_variance_f}
\mathbb E(|\widehat f_{n,\mathbf h_n}(x) - f(x)|^2) =
b_n(f,x)^2 +\textrm{var}(\widehat f_{n,\mathbf h_n}(x))
\end{equation}
where
\begin{displaymath}
b_n(f,x) :=
\mathbb E(\widehat f_{n,\mathbf h_n}(x)) - f(x).
\end{displaymath}
Let us find controls for $b_n(f,x)$ and $\textrm{var}(\widehat f_{n,\mathbf h_n}(x))$.
\\
\\
On the one hand,
\begin{eqnarray*}
 \textrm{var}(\widehat f_{n,\mathbf h_n}(x)) & = &
 \frac{1}{n^2}\sum_{k = 1}^{n}\frac{1}{h_{k}^{2}}\textrm{var}\left(K\left(\frac{X_k - x}{h_k}\right)\right)\\
 & \leqslant &
 \frac{1}{n^2}\sum_{k = 1}^{n}\frac{1}{h_{k}^{2}}
 \int_{-\infty}^{\infty}K\left(\frac{y - x}{h_k}\right)^2f(y)dy\\
 & = &
 \frac{1}{n^2}\sum_{k = 1}^{n}\frac{1}{h_k}
 \int_{-\infty}^{\infty}K(z)^2f(h_kz + x)dz\\
 & \leqslant &
 \frac{c_1}{n^2}\sum_{k = 1}^{n}\frac{1}{h_k},
\end{eqnarray*}
where
%
$c_1 \geqslant 
\|f\|_{\infty}\int_{-\infty}^{\infty}K(z)^2dz.$ Note that, from Theorem 1.1 in Tsybakov~\cite{TSYBAKOV09}, there exists a constant $c_{\max}$ such that
$$\sup_{x\in {\mathbb R}}\sup_{f\in \Sigma(\beta, L), f\geqslant 0, \int f=1}f(x)\leqslant 
c_{\max}$$
where $c_{\max}$ depends on $\beta, L$ only. Thus, we can set $c_1:=c_{\max} \|K\|^2_2$. 
%
On the other hand,
\begin{eqnarray*}
 b_n(f,x) & = &
 -f(x) +
 \frac{1}{n}\sum_{k = 1}^{n}\frac{1}{h_k}\mathbb E\left(
 K\left(\frac{X_k - x}{h_k}\right)\right)\\
 & = &
 -f(x) +
 \frac{1}{n}\sum_{k = 1}^{n}\int_{-\infty}^{\infty}
 K(z)f(h_kz + x)dz\\
 & = &
 \frac{1}{n}\sum_{k = 1}^{n}\int_{-\infty}^{\infty}
 K(z)(f(h_kz + x) - f(x))dz.
\end{eqnarray*}
For every $k\in\{ 1, \dots, n\}$ and $z\in\mathbb R$, by Taylor-Lagrange's formula there exists $\tau\in [0,1]$ such that
\begin{displaymath}
 f(h_kz + x) - f(x) =
 \sum_{i = 1}^{l - 1}\frac{(h_kz)^i}{i!}f^{(i)}(x) +
 \frac{(h_kz)^l}{l!}f^{(l)}(\tau h_kz + x).
\end{displaymath}
Then, by Assumption \ref{conditions_Kernel},
\begin{eqnarray*}
 b_n(f,x) & = &
 \frac{1}{n}\sum_{k = 1}^{n}\int_{-\infty}^{\infty}
 K(z)(f(h_kz + x) - f(x))dz\\
 & = &
 \frac{1}{n}\sum_{k = 1}^{n}\left(
 \sum_{i = 1}^{l - 1}\frac{h_{k}^{i}}{i!}f^{(i)}(x)
 \int_{-\infty}^{\infty}
 z^iK(z)dz +
 \frac{h_{k}^{l}}{l!}
 \int_{-\infty}^{\infty}z^lK(z)f^{(l)}(\tau h_kz + x)dz
 \right)\\
 & = &
 \frac{1}{n}\sum_{k = 1}^{n}\frac{h_{k}^{l}}{l!}
 \int_{-\infty}^{\infty}z^lK(z)f^{(l)}(\tau h_kz + x)dz\\
 & = &
 \frac{1}{n}\sum_{k = 1}^{n}\frac{h_{k}^{l}}{l!}
 \int_{-\infty}^{\infty}z^lK(z)(f^{(l)}(\tau h_kz + x) - f^{(l)}(x))dz.
\end{eqnarray*}
Therefore, by Assumption \ref{Holder_condition},
\begin{eqnarray*}
 |b_n(f,x)| & \leqslant &
 \frac{1}{n}\sum_{k = 1}^{n}\frac{h_{k}^{l}}{l!}
 \int_{-\infty}^{\infty}|z|^l\cdot |K(z)|\cdot |f^{(l)}(\tau h_kz + x) - f^{(l)}(x)|dz\\
 & \leqslant &
 \frac{c_2}{n}\sum_{k = 1}^{n}\frac{h_{k}^{\beta}}{l!}
\end{eqnarray*}
where
%
$c_2 :=
L
\int_{-\infty}^{\infty}|z|^{\beta}|K(z)|dz.$
%
In conclusion, by Equation (\ref{bias_variance_f}), setting $c := c_1\vee c_{2}^{2}$, we get
\begin{displaymath}
\mathbb E(|\widehat f_{n,\mathbf h_n}(x) - f(x)|^2)
\leqslant
\frac{c}{n^2}\left(\left|\sum_{k = 1}^{n}\frac{h_{k}^{\beta}}{l!}\right|^2 +
\sum_{k = 1}^{n}\frac{1}{h_k}\right). \quad \Box
\end{displaymath}
%


%
\subsection{Proof of Proposition \ref{control_MISE_f}}
In order to prove Proposition \ref{control_MISE_f}, the two following inequalities are required.
%


%
\begin{lemma}\label{generalized_Minkowski}
For any Borel function $\varphi :\mathbb R^2\rightarrow\mathbb R$, if $y\mapsto \varphi(y,z)$ is integrable and
\begin{displaymath}
y\longmapsto\int_{-\infty}^{\infty}\varphi(y,z)dz
\end{displaymath}
is a Borel function, then
\begin{enumerate}
 \item $\displaystyle{\int_{-\infty}^{\infty}\left(\int_{-\infty}^{\infty}\varphi(y,z)dy\right)^2dz
 \leqslant
 \left(\int_{-\infty}^{\infty}\left(\int_{-\infty}^{\infty}\varphi(y,z)^2dz\right)^{1/2}dy\right)^2}$.
 \item $\displaystyle{\int_{-\infty}^{\infty}\left(\sum_{k = 1}^{n}\varphi(k,z)\right)^2dz
 \leqslant
 \left(\sum_{k = 1}^{n}
 \left(\int_{-\infty}^{\infty}\varphi(k,z)^2dz\right)^{1/2}\right)^2}$.
\end{enumerate}
\end{lemma}
%


%
\begin{proof}
\noindent
Result (1) is proved in Tsybakov \cite{TSYBAKOV09}, see Lemma A.1 for a proof. Result (2) follows from triangular inequality applied to the $\mathbb L^2$-norm of the sum of $n$ functions.
\end{proof}
\noindent
It has been established in the proof of Proposition \ref{control_MSE_f} that
\begin{equation}\label{control_MISE_f_1}
\textrm{var}(\widehat f_{n,\mathbf h_n}(x))
\leqslant
\frac{1}{n^2}
\sum_{k = 1}^{n}
\frac{1}{h_{k}^{2}}\int_{-\infty}^{\infty}K\left(\frac{y - x}{h_k}\right)^2f(y)dy
\end{equation}
and
\begin{equation}\label{control_MISE_f_2}
b_n(f,x) =
\frac{1}{n}\sum_{k = 1}^{n}\int_{-\infty}^{\infty}
 K(z)(f(h_kz + x) - f(x))dz.
\end{equation}
On the one hand, by Inequality (\ref{control_MISE_f_1}),
\begin{eqnarray*}
 \int_{-\infty}^{\infty}
 \textrm{var}(\widehat f_{n,\mathbf h_n}(x))dx & \leqslant &
 \frac{1}{n^2}
 \sum_{k = 1}^{n}
 \frac{1}{h_{k}^{2}}\int_{-\infty}^{\infty}f(y)
 \int_{-\infty}^{\infty}K\left(\frac{y - x}{h_k}\right)^2dxdy\\
 & = &
 \frac{1}{n^2}
 \sum_{k = 1}^{n}
 \frac{1}{h_k}\left(\int_{-\infty}^{\infty}f(y)dy\right)\left(
 \int_{-\infty}^{\infty}K(z)^2dz\right)\\
 & \leqslant &
 \frac{c_1}{n^2}
 \sum_{k = 1}^{n}
 \frac{1}{h_k} 
\end{eqnarray*}
where
%
$c_1 := \|K\|^2_2$.
%
On the other hand, by Taylor's formula with integral remainder,
\begin{displaymath}
f(h_kz + x) - f(x) =
\sum_{i = 1}^{l - 1}\frac{(h_kz)^i}{i!}f^{(i)}(x) +
\frac{(h_kz)^l}{(l - 1)!}
\int_{0}^{1}(1 -\tau)^{l - 1}
f^{(l)}(\tau h_kz + x)d\tau.
\end{displaymath}
Then, by Assumption \ref{conditions_Kernel},
\begin{eqnarray*}
 b_n(f,x) & = &
 \frac{1}{n}\sum_{k = 1}^{n}\int_{-\infty}^{\infty}
 K(z)(f(h_kz + x) - f(x))dz\\
 & = &
 \frac{1}{n}\sum_{k = 1}^{n}
 \sum_{i = 1}^{l - 1}\frac{h_{k}^{i}}{i!}f^{(i)}(x)
 \int_{-\infty}^{\infty}
 z^iK(z)dz\\
 & &
 +\frac{1}{n}\sum_{k = 1}^{n}
 \frac{h_{k}^{l}}{(l - 1)!}
 \int_{-\infty}^{\infty}z^lK(z)
 \int_{0}^{1}(1 -\tau)^{l - 1}f^{(l)}(\tau h_kz + x)d\tau dz\\
 & = &
 \frac{1}{n}\sum_{k = 1}^{n}
 \frac{h_{k}^{l}}{(l - 1)!}
 \int_{-\infty}^{\infty}z^lK(z)
 \int_{0}^{1}(1 -\tau)^{l - 1}(f^{(l)}(\tau h_kz + x) - f^{(l)}(x))d\tau dz.
\end{eqnarray*}
By Lemma \ref{generalized_Minkowski}.(2),
\begin{displaymath}
\int_{-\infty}^{\infty}b_n(f,x)^2dx\leqslant
\frac{1}{n^2}
\left(
\sum_{k = 1}^{n}
\frac{h_{k}^{l}}{(l - 1)!}u_{k}^{1/2}\right)^2
\end{displaymath}
where
\begin{displaymath}
u_k :=
\int_{-\infty}^{\infty}\left|
\int_{-\infty}^{\infty}z^lK(z)
\int_{0}^{1}(1 -\tau)^{l - 1}(f^{(l)}(\tau h_kz + x) - f^{(l)}(x))d\tau dz\right|^2dx.
\end{displaymath}
By Lemma \ref{generalized_Minkowski}.(1), for every $k\in \{1, \dots, l\}$,
\begin{displaymath}
u_k
\leqslant
\left(\int_{-\infty}^{\infty}|z|^l|K(z)|\int_{0}^{1}(1 - \tau)^{l - 1}
\left(\int_{-\infty}^{\infty}
|f^{(l)}(\tau h_k z + x) - f^{(l)}(x)|^2dx\right)^{1/2}d\tau dz\right)^2.
\end{displaymath}
Therefore, by Assumption \ref{Nikolski_condition},
\begin{eqnarray*}
 \int_{-\infty}^{\infty}b_n(f,x)^2dx & \leqslant &
 \frac{c_2}{n^2}
 \left(
 \sum_{k = 1}^{n}
 \frac{h_{k}^{\beta}}{(l - 1)!}\right)^2
\end{eqnarray*}
where
\begin{displaymath}
c_2 :=
L^2\left(\int_{-\infty}^{\infty}|z|^{\beta}|K(z)|dz\right)^2
\end{displaymath}
since $\left(\int_{0}^{1}(1 -\tau)^{l - 1}\tau^{\beta - l}d\tau\right)^2\leqslant 1$. 
In conclusion, by Equation (\ref{bias_variance_f}), setting $c := c_1\vee c_2$, we get
\begin{displaymath}
\int_{-\infty}^{\infty}
\mathbb E(|\widehat f_{n,\mathbf h_n}(x) - f(x)|^2)dx
\leqslant
\frac{c}{n^2}\left(\left|\sum_{k = 1}^{n}\frac{h_{k}^{\beta}}{(l - 1)!}\right|^2 +
\sum_{k = 1}^{n}\frac{1}{h_k}\right). \quad \Box
\end{displaymath}
%


%
\subsection{Proof of Proposition \ref{Goldenshluger_Lepski}}
For any $\gamma$ in $\Gamma_n$, we have 
\begin{eqnarray*}
 \|\widehat f_{n,\mathbf h_n(\widehat\gamma_n)} - f\|_{2}^{2}
 & \leqslant &
3( \|\widehat f_{n,\mathbf h_n(\widehat\gamma_n)} -
 \widehat f_{n,\widehat\gamma_n,\gamma}\|_{2}^{2}  +
 \|\widehat f_{n,\mathbf h_n(\gamma)} -
 \widehat f_{n,\widehat\gamma_n,\gamma}\|_{2}^{2}\\
&& +
 \|\widehat f_{n,\mathbf h_n(\gamma)} - f\|_{2}^{2}).
\end{eqnarray*}
By the definition of $A_n$,
\begin{eqnarray*}
 \|\widehat f_{n,\mathbf h_n(\widehat\gamma_n)} -
 \widehat f_{n,\widehat\gamma_n,\gamma}\|_{2}^{2} & = &
 \|\widehat f_{n,\mathbf h_n(\widehat\gamma_n)} -
 \widehat f_{n,\gamma,\widehat\gamma_n}\|_{2}^{2}\\
 & \leqslant &
 A_n(\gamma) + V_n(\widehat\gamma_n)
\end{eqnarray*}
and
\begin{displaymath}
\|\widehat f_{n,\mathbf h_n(\gamma)} -
\widehat f_{n,\widehat\gamma_n,\gamma}\|_{2}^{2}
\leqslant
A_n(\widehat\gamma_n) + V_n(\gamma).
\end{displaymath}
Thus we get that, for any $\gamma\in\Gamma_n$,
\begin{eqnarray*}
 \|\widehat f_{n,\mathbf h_n(\widehat\gamma_n)} - f\|_{2}^{2} \leqslant 3( A_n(\gamma) + V_n(\widehat\gamma_n) +  A_n(\widehat\gamma_n) + V_n(\gamma)+
  \|\widehat f_{n,\mathbf h_n(\gamma)} - f\|_{2}^{2}).
\end{eqnarray*}
Since
\begin{displaymath}
\widehat\gamma_n\in
\arg\min_{\gamma\in\Gamma_n}
(A_n(\gamma) + V_n(\gamma)),
\end{displaymath}
it follows that, for any $\gamma\in\Gamma_n$,
\begin{eqnarray}  \label{Goldenshluger_Lepski_1}
 \mathbb E(\|\widehat f_{n,\mathbf h_n(\widehat\gamma_n)} - f\|_{2}^{2})
 & \leqslant &
3\mathbb E(\|\widehat f_{n,\mathbf h_n(\gamma)} - f\|_{2}^{2}) +
 6[\mathbb E(A_n(\gamma)) + V_n(\gamma)]. 
\end{eqnarray}
Now, let us find a suitable control for $\mathbb E(A_n(\gamma))$. For that, consider
\begin{displaymath}
f_{n,\gamma} :=
\frac{1}{n}\sum_{k = 1}^{n}K_{h_k(\gamma)}\ast f \quad \mbox{ and } \quad
f_{n,\gamma,\gamma'} :=
\frac{1}{n}\sum_{k = 1}^{n}
K_{h_k(\gamma')}\ast
K_{h_k(\gamma)}\ast f,
\end{displaymath}
with obviously $f_{n,\gamma}(x) = {\mathbb E}(\widehat f_{n,\mathbf h_n(\gamma)}(x))$ and $f_{n,\gamma,\gamma'}(x) = {\mathbb E}(\widehat f_{n,\gamma,\gamma'}(x))$.
Now,  for any $\gamma, \gamma'\in\Gamma_n$,
\begin{eqnarray*}
 \|\widehat f_{n,\mathbf h_n(\gamma')} -
 \widehat f_{n,\gamma,\gamma'}\|_{2}^{2}
 & \leqslant &
3( \|\widehat f_{n,\mathbf h_n(\gamma')} -
 f_{n,\gamma'}\|_{2}^{2}  +
 \|f_{n,\gamma'} - f_{n,\gamma,\gamma'}\|_{2}^{2}
 +
 \|\widehat f_{n,\gamma,\gamma'} -
 f_{n,\gamma,\gamma'}\|_{2}^{2}).
\end{eqnarray*}
Thus, using the definition of $A_n(\gamma)$, we get that for any $\gamma\in \Gamma_n$, 
\begin{eqnarray}
 \label{Goldenshluger_Lepski_2}
 A_n(\gamma)
 & \leqslant &
 3\left(\sup_{\gamma'\in\Gamma_n}\left(
 \|\widehat f_{n,\mathbf h_n(\gamma')} -
 f_{n,\gamma'}\|_{2}^{2} -
 \frac{V_n(\gamma')}{6}\right)_+\right.\\
 & &
 +\left.
 \sup_{\gamma'\in\Gamma_n}\left(
 \|\widehat f_{n,\gamma,\gamma'} -
 f_{n,\gamma,\gamma'}\|_{2}^{2} -
 \frac{V_n(\gamma')}{6}\right)_+ +
 \|f_{n,\gamma'} - f_{n,\gamma,\gamma'}\|_{2}^{2}\right).
 \nonumber
\end{eqnarray}
Let us control each terms of the right-hand side of Inequality (\ref{Goldenshluger_Lepski_2}). On the one hand, by Lemma \ref{generalized_Minkowski}.(2),
\begin{eqnarray} \nonumber
 \|f_{n,\gamma'} -
 f_{n,\gamma,\gamma'}\|_{2}^{2} & = &
 \left\|\frac{1}{n}\sum_{k = 1}^{n}K_{h_k(\gamma')}\ast (f - K_{h_k(\gamma)}\ast f)\right\|_{2}^{2}\\ \nonumber 
 & \leqslant &
 \frac{1}{n^2}\left|\sum_{k = 1}^{n}\|K_{h_k(\gamma')}\ast (f - K_{h_k(\gamma)}\ast f)\|_2\right|^2\\ \label{Goldenshluger_Lepski_3}
 & \leqslant &
 \frac{\|K\|_{1}^{2}}{n^2}\left|
 \sum_{k = 1}^{n}\|f - K_{h_k(\gamma)}\ast f\|_2\right|^2.
\end{eqnarray}
On the other hand, let $\mathcal C$ be a countable and dense subset of the unit sphere of $\mathbb L^2(\mathbb R,dx)$. Then,
\begin{small}
\begin{displaymath}
 \mathbb E\left(\sup_{\gamma'\in\Gamma_n}\left(
 \|\widehat f_{n,\mathbf h_n(\gamma')} -
 f_{n,\gamma'}\|_{2}^{2} -
 \frac{V_n(\gamma')}{6}\right)_+\right)\leqslant
 \sum_{\gamma'\in\Gamma_n}\mathbb E\left(\left(
 \sup_{\psi\in\mathcal C}\mathfrak v_{n,\gamma'}(\psi)^2 -
 \frac{V_n(\gamma')}{6}\right)_+\right)
\end{displaymath}
\end{small}
\newline
where, for every $\psi\in\mathcal C$,
\begin{eqnarray*}
 \mathfrak v_{n,\gamma'}(\psi) & := &
 \langle\psi,\widehat f_{n,\mathbf h_n(\gamma')} -
 f_{n,\gamma'}\rangle_2\\
 & = &
 \frac{1}{n}\sum_{k = 1}^{n}(
 v_{\psi}(h_k(\gamma'),X_k) -\mathbb E(v_{\psi}(h_k(\gamma'),X_k)))
\end{eqnarray*}
and
\begin{displaymath}
\upsilon_{\psi}(h,y) :=
\int_{-\infty}^{\infty}\psi(x)K_h(y - x)dx
\textrm{ $;$ }
\forall (h,y)\in (1/n,1)\times\mathbb R.
\end{displaymath}
In order to apply Talagrand's inequality (see Klein and Rio \cite{KR05}):
\begin{enumerate}
 \item For every $\psi\in\mathcal C$, $h\in (1/n,1)$ and $y\in\mathbb R$,
 \begin{eqnarray*}
  |v_{\psi}(h,y)| & \leqslant &
  \int_{-\infty}^{\infty}|\psi(x)|K_h(y - x)dx\\
  & \leqslant &
  \|K_h(y -\cdot)\|_2 =
  \frac{1}{\sqrt h}\|K\|_2
  \leqslant\|K\|_2\sqrt n.
 \end{eqnarray*}
 Then,
 \begin{displaymath}
 \sup_{\psi\in\mathcal C}
 \|v_{\psi}\|_{\infty}
 \leqslant M_1(n) :=\|K\|_2\sqrt n.
 \end{displaymath}
 \item For every $\psi\in\mathcal C$,
 \begin{eqnarray*}
  \mathfrak v_{n,\gamma'}(\psi)^2 & = &
  \langle\psi,\widehat f_{n,\mathbf h_n(\gamma')} - f_{n,\gamma'}\rangle_{2}^{2}\\
  & \leqslant &
  \|\widehat f_{n,\mathbf h_n(\gamma')} - f_{n,\gamma'}\|_{2}^{2}
  =\int_{-\infty}^{\infty}|\widehat f_{n,\mathbf h_n(\gamma')}(x) -\mathbb E(\widehat f_{n,\mathbf h_n(\gamma')}(x))|^2dx.
 \end{eqnarray*}
 Then, as established in the proof of Proposition \ref{control_MISE_f},
 \begin{eqnarray*}
  \mathbb E\left(\sup_{\psi\in\mathcal C}|\mathfrak v_{n,\gamma'}(\psi)|\right) & \leqslant &
  \left|\int_{-\infty}^{\infty}\textrm{var}(\widehat f_{n,\mathbf h_n(\gamma')}(x))dx\right|^{1/2}\\
  & \leqslant &
  \frac{\|K\|_2}{n}\left|\sum_{k = 1}^{n}\frac{1}{h_k(\gamma')}\right|^{1/2} \leqslant   M_2(n,\gamma') :=
  \frac{V_n(\gamma')^{1/2}}{\upsilon^{1/2}},
 \end{eqnarray*}
as $\|K\|_1\geqslant 1=|\int K|$.
  \item For every $\psi\in\mathcal C$ and $k\in\{1, \dots, n\}$,
 \begin{eqnarray*}
  \textrm{var}(v_{\psi}(h_k(\gamma'),X_k))
  & \leqslant &
  \mathbb E\left(\left|
  \int_{-\infty}^{\infty}
  K_{h_k(\gamma')}(X_k - x)\psi(x)dx\right|^2\right)\\
  & = &
  \int_{-\infty}^{\infty}
  (K_{h_k(\gamma')}\ast\psi)(y)^2f(y)dy\leqslant
  \|f\|_{\infty}\|K\|_{1}^{2}.
 \end{eqnarray*}
 Then,
 \begin{displaymath}
 \sup_{\psi\in\mathcal C}
 \frac{1}{n}\sum_{k = 1}^{n}
 \textrm{var}(v_{\psi}(h_k(\gamma'),X_k))
 \leqslant
 M_3 :=
 \|f\|_{\infty}\|K\|_{1}^{2}.
 \end{displaymath}
\end{enumerate}
By applying Talagrand's inequality to $(v_{\psi})_{\psi\in\mathcal C}$ and to the independent random variables $(h_1(\gamma'),X_1),\dots,(h_n(\gamma'),X_n)$, there exist two numerical constants $c_1,c_2 > 0$, and two constant $c_3, c_4 > 0$ depending only on $f$, $K$ and $\upsilon$, such that
\begin{eqnarray*}
 \mathbb E\left(\left(\sup_{\psi\in\mathcal C}
 \mathfrak v_{n,\gamma'}(\psi)^2 - 4M_2(n,\gamma')^2\right)_+\right)
 & \leqslant &
 c_1\left(\frac{M_3}{n}\exp\left(
 -\frac{c_2}{M_3}
 nM_2(n,\gamma')^2\right)\right.\\
 & &
 +\left.
 \frac{M_1(n)^2}{n^2}\exp\left(-c_2\frac{nM_2(n,\gamma')}{M_1(n)}\right)\right)\\
 & \leqslant &
 \frac{c_3}{n}(
 \exp(-c_4/\mathfrak h_n(\gamma')) +
 \exp(-c_4/\mathfrak h_n(\gamma')^{1/2})).
\end{eqnarray*}
Then, by Assumption \ref{assumption_Gamma_n}, with $\upsilon \geqslant 24$, there exists a constant $c_5 > 0$, not depending on $n$, such that
\begin{equation}\label{Goldenshluger_Lepski_4}
\mathbb E\left(\sup_{\gamma'\in\Gamma_n}\left(
\|\widehat f_{n,\mathbf h_n(\gamma')} -
f_{n,\gamma'}\|_{2}^{2} -
\frac{V_n(\gamma')}{6}\right)_+\right)
\leqslant
\frac{c_5}{n}.
\end{equation}
The same ideas give that there exists a constant $c_{6} > 0$, not depending on $n$, such that
\begin{equation}\label{Goldenshluger_Lepski_5}
\sup_{\gamma\in\Gamma_n}
\mathbb E\left(\sup_{\gamma'\in\Gamma_n}\left(
\|\widehat f_{n,\gamma,\gamma'} -
f_{n,\gamma,\gamma'}\|_{2}^{2} -
\frac{V_n(\gamma')}{6}\right)_+\right)
\leqslant
\frac{c_{6}}{n}.
\end{equation}
By Inequalities (\ref{Goldenshluger_Lepski_2})-(\ref{Goldenshluger_Lepski_5}), we get that, for all $\gamma\in \Gamma_n$,
$${\mathbb E}(A_n(\gamma))\leqslant 3 \left(  \frac{\|K\|_{1}^{2}}{n^2}\left|
 \sum_{k = 1}^{n}\|f - K_{h_k(\gamma)}\ast f\|_2\right|^2 + \frac{c_5+c_6}n\right).$$
Plugging this in Inequality (\ref{Goldenshluger_Lepski_1}), using that
$$\mathbb E(\|\widehat f_{n,\mathbf h_n(\gamma)} - f\|_{2}^{2})\leqslant 
\frac{1}{n^2}\left|
 \sum_{k = 1}^{n}\|f - K_{h_k(\gamma)}\ast f\|_2\right|^2 + \frac{1}{\upsilon}  V_n(\gamma),$$
 we get
\begin{equation}\label{GL7}
\mathbb E(\|\widehat f_{n,\mathbf h_n(\widehat\gamma_n)} - f\|_{2}^{2})
\leqslant
c_7
\inf_{\gamma\in\Gamma_n}
\left\{V_n(\gamma) + 
\frac{1}{n^2}\left|\sum_{k = 1}^{n}\|f - K_{h_k(\gamma)}\ast f\|_2\right|^2\right\} +
\frac{c_8}{n}
\end{equation}
where $c_{7}=3 \max( 2+\frac 1{\upsilon},1+6\|K\|_1^2)$ and $c_8 > 0$ is not depending on $n$. Note that, as $\upsilon\geqslant 24$ and $\|K\|_1^2\geqslant 1$, $c_7=18 \|K\|_1^2$. We obtain the first inequality of Proposition \ref{Goldenshluger_Lepski}.
\\
\\
To get (\ref{GLresult}), we write
\begin{displaymath}
V_n(\gamma) =
\upsilon\frac{\|K\|_2^2\|K\|_1^2}{n\, \mathfrak h_n(\gamma)} =
\upsilon \|K\|_2^2\|K\|_1^2 \mathbb V_n(\gamma)
\end{displaymath}
and, if Assumptions \ref{conditions_Kernel} and \ref{Nikolski_condition} hold, as established in the proof of Proposition \ref{control_MISE_f}, there exists a constant $c_{9} > 0$ which does not depend on $n$ such that
\begin{eqnarray*}
 \sum_{k = 1}^{n}\|f - K_{h_k(\gamma)}\ast f\|_2
 & = &
 \sum_{k = 1}^{n}\left(\int_{-\infty}^{\infty}\left|\int_{-\infty}^{\infty}K(z)(f(h_k(\gamma)z + x) - f(x))dz\right|^2dx\right)^{1/2}\\
 & \leqslant &
 c_{9}\sum_{k = 1}^{n}h_k(\gamma)^{\beta}
 = c_{9}n\mathbb B_n(\gamma)^{1/2}.
\end{eqnarray*}
Plugging these last two bounds in (\ref{GL7}) yields (\ref{GLresult}) and ends the proof of Proposition~\ref{Goldenshluger_Lepski}.
%


%
\subsection{Proof of Proposition \ref{LMR_method}}
For this proof, we use the tools and follow the lines given in the proof of Theorem 2 in Lacour {\it et al.} \cite{LMR17}.
\\
\\
Throughout this section, for every $h > 0$, we consider $f_h := f\ast K_h$, where $\ast$ is the convolution product and we recall that $K_h = 1/hK(\cdot /h)$. Note that for every $h > 0$ and $k\in\{1, \dots, n\}$,
\begin{displaymath}
\mathbb E(K_h(X_k - x)) =
\int_{-\infty}^{\infty}K_h(y - x)f(y)dy =
f_h(x).
\end{displaymath}
We also consider $\lambda\in [1,\infty[$, $\varepsilon,\theta\in (0,1)$ and $\gamma\in\Gamma_n$.
\\
\\
In order to prove Proposition \ref{LMR_method}, let us first establish the three following lemmas providing suitable bounds for key quantities involved in the decomposition of
\begin{displaymath}
\|\widehat f_{n,\mathbf h_n(\widetilde\gamma_n)} - f\|_{2}^{2}.
\end{displaymath}
Throughout this section, the positive constants $\kappa_i$ ; $i\in\mathbb N^*$ are not depending on $n$, $\lambda$, $\theta$, $\varepsilon$ and $\gamma$.
%


%
\subsubsection{Steps of the proof}
The proof relies on three Lemmas, which are stated first. Lemma \ref{U_statistic_bound} follows from an exponential inequality for $U$-statistics, which is applied here in a non identically distributed context. 
%


%
\begin{lemma}\label{U_statistic_bound}
Consider the $U$-statistic
\begin{displaymath}
U_n(\gamma,\gamma_{\max}) :=
\sum_{k\not= l}\langle K_{h_k(\gamma)}(X_k -\cdot) - f_{h_k(\gamma)},
K_{h_l(\gamma_{\max})}(X_l -\cdot) - f_{h_l(\gamma_{\max})}
\rangle_2.
\end{displaymath}
There exists a universal constant $\mathfrak c > 0$ such that with probability larger than $1 - 5.54|\Gamma_n|e^{-\lambda}$, 
\begin{displaymath}
\frac{|U_n(\gamma,\gamma_{\max})|}{n^2}
\leqslant
\frac{\theta\|K\|_{2}^{2}}{n\mathfrak h_n(\gamma)} +
\frac{\mathfrak c}{\theta}\left(\frac{\|K\|_{1}^{2}\|f\|_{\infty}}{n}\lambda^2 +
\frac{\|K\|_1\|K\|_{\infty}}{n^2h_n(\gamma_{\max})}\lambda^3\right).
\end{displaymath}
\end{lemma}
\noindent
Lemmas \ref{bound_V}  and \ref{Lerasle_type_inequality} rely on Bernstein's inequality for non identically distributed variables.
%


%
\begin{lemma}\label{bound_V}
There exists a deterministic constant $c > 0$, not depending on $n$, $\lambda$, $\theta$ and $\gamma$, such that for every $\gamma'\in\Gamma_n$, with probability larger than $1 - 2e^{-\lambda}$,
\begin{displaymath}
V_n(\gamma,\gamma') :=
\langle\widehat f_{n,\mathbf h_n(\gamma)} - f_{n,\gamma},
 f_{n,\gamma'} - f\rangle
 \end{displaymath}
satisfies 
\begin{displaymath}
|V_n(\gamma, \gamma')|
\leqslant
\theta\|f_{n,\gamma'} - f\|_{2}^{2} +\frac{c\lambda}{\theta n}.
\end{displaymath}
\end{lemma}
%


%
\begin{lemma}\label{Lerasle_type_inequality}
Under Assumption \ref{conditions_LMR}, there exists two deterministic constants $c_1,c_2 > 0$, not depending on $n$, $\lambda$, $\varepsilon$ and $\gamma$, such that with probability larger than $1 - c_1|\Gamma_n|e^{-\lambda}$,
\begin{displaymath}
\|f_{n,\gamma} - f\|_{2}^{2} +\frac{\|K\|_{2}^{2}}{n \,\mathfrak h_n(\gamma)}\leqslant
(1 +\varepsilon)\|\widehat f_{n,\mathbf h_n(\gamma)} - f\|_{2}^{2} +
c_2\frac{(1+\varepsilon)^2}{\varepsilon}\left(\frac{\lambda^2}{n} +
\frac{\lambda^3}{n^2h_n(\gamma_{\max})}\right).
\end{displaymath}
\end{lemma}
%


%
%
\noindent
The proof of Proposition \ref{LMR_method} is dissected in three steps.
\\

\noindent \textbf{Step 1.} In this step, a suitable decomposition of
\begin{displaymath}
\|\widehat f_{n,\mathbf h_n(\widetilde\gamma_n)} - f\|_{2}^{2}
\end{displaymath}
is provided. On the one hand,
\begin{eqnarray*}
 \|\widehat f_{n,\mathbf h_n(\widetilde\gamma_n)} - f\|_{2}^{2} +\textrm{pen}(\widetilde\gamma_n) & = &
 \|\widehat f_{n,\mathbf h_n(\widetilde\gamma_n)} -\widehat f_{n,\mathbf h_n(\gamma_{\max})}\|_{2}^{2}
 +\textrm{pen}(\widetilde\gamma_n)\\
 & &
 +\|\widehat f_{n,\mathbf h_n(\gamma_{\max})} - f\|_{2}^{2}\\
 & &
 - 2\langle\widehat f_{n,\mathbf h_n(\gamma_{\max})} - f,
 \widehat f_{n,\mathbf h_n(\gamma_{\max})} -\widehat f_{n,\mathbf h_n(\widetilde\gamma_n)}\rangle_2.
\end{eqnarray*}
Since $\widetilde\gamma_n\in\arg\min_{\gamma\in\Gamma_n}\textrm{Crit}(\gamma)$, for any $\gamma\in \Gamma_n$, 
\begin{eqnarray}
 \|\widehat f_{n,\mathbf h_n(\widetilde\gamma_n)} - f\|_{2}^{2} & \leqslant &
 \|\widehat f_{n,\mathbf h_n(\gamma)} -\widehat f_{n,\mathbf h_n(\gamma_{\max})}\|_{2}^{2}
 +\textrm{pen}(\gamma)\nonumber\\
 & &
 -\textrm{pen}(\widetilde\gamma_n)
 +\|\widehat f_{n,\mathbf h_n(\gamma_{\max})} - f\|_{2}^{2}\nonumber\\
 & &
 - 2\langle\widehat f_{n,\mathbf h_n(\gamma_{\max})} - f,
 \widehat f_{n,\mathbf h_n(\gamma_{\max})} -\widehat f_{n,\mathbf h_n(\widetilde\gamma_n)}\rangle_2\nonumber\\
 & = &
 \|\widehat f_{n,\mathbf h_n(\gamma)} - f\|_{2}^{2}\nonumber\\
 & &
 -[\textrm{pen}(\widetilde\gamma_n)
 - 2\|\widehat f_{n,\mathbf h_n(\gamma_{\max})} - f\|_{2}^{2}\nonumber\\
 & &
 + 2\langle\widehat f_{n,\mathbf h_n(\gamma_{\max})} - f,
 \widehat f_{n,\mathbf h_n(\gamma_{\max})} -\widehat f_{n,\mathbf h_n(\widetilde\gamma_n)}\rangle_2]\nonumber\\
 & &
 +\textrm{pen}(\gamma)
 - 2\langle\widehat f_{n,\mathbf h_n(\gamma_{\max})} - f,\widehat f_{n,\mathbf h_n(\gamma)} - f\rangle_2\nonumber\\
 \label{Oracle_inequality_1}
 & = &
 \|\widehat f_{n,\mathbf h_n(\gamma)} - f\|_{2}^{2} +
 \textrm{pen}(\gamma) - 2\psi_n(\gamma)
 -(\textrm{pen}(\widetilde\gamma_n) - 2\psi_n(\widetilde\gamma_n))
\end{eqnarray}
with
\begin{displaymath}
\psi_n :=\langle\widehat f_{n,\mathbf h_n(\gamma_{\max})} - f,
\widehat f_{n,\mathbf h_n(.)} - f\rangle_2.
\end{displaymath}
On the other hand,
\begin{eqnarray*}
 \psi_n(\gamma) & = &
 \langle\widehat f_{n,\mathbf h_n(\gamma_{\max})} - f_{n,\gamma_{\max}},
 \widehat f_{n,\mathbf h_n(\gamma)} - f_{n,\gamma}\rangle_2 
 +\langle\widehat f_{n,\mathbf h_n(\gamma_{\max})} - f_{n,\gamma_{\max}},
 f_{n,\gamma} - f\rangle_2 \\ &&
 +\langle f_{n,\gamma_{\max}} - f,
 \widehat f_{n,\mathbf h_n(\gamma)} - f_{n,\gamma}\rangle_2
 +\langle  f_{n,\gamma_{\max}} - f,
  f_{n,\gamma} - f\rangle_2\\
 & = &
 \psi_{1,n}(\gamma) +
 \psi_{2,n}(\gamma) +
 \psi_{3,n}(\gamma),
\end{eqnarray*}
where
\begin{eqnarray*}
 \psi_{1,n}(\gamma) & := &
 \frac{1}{n^2}\sum_{k = 1}^{n}\langle K_{h_k(\gamma)},K_{h_k(\gamma_{\max})}\rangle_2 +
 \frac{U_n(\gamma,\gamma_{\max})}{n^2}\textrm{, }\\
 \psi_{2,n}(\gamma) & := &
 \frac{1}{n^2}\left(-\sum_{k = 1}^{n}\langle K_{h_k(\gamma)}(X_k -\cdot),f_{h_k(\gamma_{\max})}\rangle_2\right.\\
 & &
 \left.-\sum_{k = 1}^{n}\langle K_{h_k(\gamma_{\max})}(X_k -\cdot),f_{h_k(\gamma)}\rangle_2
 +\sum_{k = 1}^{n}\langle f_{h_k(\gamma)},f_{h_k(\gamma_{\max})}\rangle_2\right)\textrm{ and}\\
 \psi_{3,n}(\gamma) & := &
 V_n(\gamma,\gamma_{\max}) +
 V_n(\gamma_{\max},\gamma) +\langle f_{n,\gamma} - f,f_{n,\gamma_{\max}} - f\rangle_2
\end{eqnarray*}
\textbf{Step 2.} Some bounds for $\psi_{n,1}(\gamma)$, $\psi_{n,2}(\gamma)$ and $\psi_{n,3}(\gamma)$ are provided in this step.
\begin{enumerate}
 \item Consider
 \begin{displaymath}
 \widetilde\psi_{1,n}(\gamma) :=
 \psi_{1,n}(\gamma) -\frac{1}{n^2}\sum_{k = 1}^{n}\langle K_{h_k(\gamma)},K_{h_k(\gamma_{\max})}\rangle_2.
 \end{displaymath}
 By Lemma \ref{U_statistic_bound}, with probability larger than $1 - 5.54|\Gamma_n|e^{-\lambda}$,
 \begin{eqnarray*}
  |\widetilde\psi_{1,n}(\gamma)|
  & = &
  \frac{|U_n(\gamma,\gamma_{\max})|}{n^2}\\
  & \leqslant &
  \frac{\theta\|K\|_{2}^{2}}{n \,\mathfrak h_n(\gamma)} +
  \frac{\mathfrak c}{\theta}\left(\frac{\|K\|_{1}^{2}\|f\|_{\infty}}{n}\lambda^2 +
  \frac{\|K\|_{\infty}\|K\|_1}{n^2h_n(\gamma_{\max})}\lambda^3\right).
 \end{eqnarray*}
 \item On the one hand, for any $\gamma'\in\Gamma_n$,
 \begin{eqnarray*}
  \frac{1}{n}\left|\sum_{k = 1}^{n}
  \langle K_{h_k(\gamma)}(X_k -\cdot),f_{h_k(\gamma')}\rangle_2\right| & \leqslant &
  \max_{k\in\{1, \dots, n\}}
  \int_{-\infty}^{\infty}|K_{h_k(\gamma)}(X_k - x)f_{h_k(\gamma')}(x)|dx\\
  & \leqslant &
  \|K\|_1\max_{k\in\{1, \dots, n\}}
  \|K_{h_k(\gamma')}\ast f\|_{\infty}
  \leqslant\|K\|_{1}^{2}\|f\|_{\infty}.
 \end{eqnarray*}
 On the other hand,
 \begin{eqnarray*}
  \frac{1}{n}\left|\sum_{k = 1}^{n}
  \langle f_{h_k(\gamma)},f_{h_k(\gamma')}\rangle_2\right| & \leqslant &
  \max_{k\in\{1, \dots, n\}}
  \int_{-\infty}^{\infty}|f_{h_k(\gamma)}(x)f_{h_k(\gamma')}(x)|dx\\
  & \leqslant &
  \max_{k\in\{1, \dots, n\}}
  \|K_{h_k(\gamma)}\ast f\|_1
  \|K_{h_k(\gamma')}\ast f\|_{\infty}
  \leqslant\|K\|_{1}^{2}\|f\|_{\infty}.
 \end{eqnarray*}
 Therefore,
 \begin{displaymath}
 \|\psi_{2,n}\|_{\infty}
 \leqslant
 \frac{3\|K\|_{1}^{2}\|f\|_{\infty}}{n}.
 \end{displaymath}
 \item By applying Lemma \ref{bound_V} to $V_n(\gamma,\gamma_{\max})$ and $V_n(\gamma_{\max},\gamma)$, with probability larger than $1-2e^{-\lambda}$, 
 \begin{eqnarray*}
  |\psi_{n,3}(\gamma)| & \leqslant &
  \frac{\theta}{2}(\|f_{n,\gamma} - f\|_{2}^{2} +\|f_{n,\gamma_{\max}} - f\|_{2}^{2}) +
  \frac{\kappa_1\lambda}{\theta n}\\
  & &
  +\frac{\theta}{2}\|f_{n,\gamma} - f\|_{2}^{2} +
  \frac{1}{2\theta}\|f_{n,\gamma_{\max}} - f\|_{2}^{2}\\
  & \leqslant &
  \theta\|f_{n,\gamma} - f\|_{2}^{2} +\left(\frac{\theta}{2} +\frac{1}{2\theta}\right)\|f_{n,\gamma_{\max}} - f\|_{2}^{2} +
  \frac{\kappa_1\lambda}{\theta n}.
 \end{eqnarray*}
\end{enumerate}
\textbf{Step 3.} Consider
\begin{displaymath}
\widetilde\psi_n(\gamma) :=
\psi_n(\gamma) -\frac{1}{n^2}\sum_{k = 1}^{n}\langle K_{h_k(\gamma)},K_{h_k(\gamma_{\max})}\rangle_2.
\end{displaymath}
By Step 2 and Lemma \ref{Lerasle_type_inequality}, with probability larger than $1 - \kappa_2|\Gamma_n|e^{-\lambda}$,
\begin{eqnarray*}
 |\widetilde\psi_n(\gamma)| & \leqslant &
 \theta\|f_{n,\gamma} - f\|_{2}^{2} +
 \frac{\theta\|K\|_{2}^{2}}{n \,\mathfrak h_n(\gamma)}\\
 & &
 +\left(\frac{\theta}{2} +\frac{1}{2\theta}\right)\|f_{n,\gamma_{\max}} - f\|_{2}^{2} +
 \frac{\kappa_3}{\theta}\left(\frac{\lambda^2}{n} +
 \frac{\lambda^3}{n^2h_n(\gamma_{\max})}\right)\\
 & \leqslant &
 2\theta\|\widehat f_{n,\mathbf h_n(\gamma)} - f\|_{2}^{2}\\
 & &
 +\left(\frac{\theta}{2} +\frac{1}{2\theta}\right)\|f_{n,\gamma_{\max}} - f\|_{2}^{2} +
 \frac{\kappa_4}{\theta}\left(\frac{\lambda^2}{n} +
 \frac{\lambda^3}{n^2h_n(\gamma_{\max})}\right).
\end{eqnarray*}
Therefore, choosing $\theta$ as in Lemma \ref{Lerasle_type_inequality} ($1/(1-\theta)=1+\varepsilon$), by Inequality (\ref{Oracle_inequality_1}), with probability larger than $1 - \kappa_5|\Gamma_n|e^{-\lambda}$,
\begin{eqnarray*}
 \|\widehat f_{n,\mathbf h_n(\widetilde\gamma_n)} - f\|_{2}^{2}
 & \leqslant &
 (1 +\varepsilon)\|\widehat f_{n,\mathbf h_n(\gamma)} - f\|_{2}^{2}
 +\frac{\kappa_6}{\varepsilon}\|f_{n,\gamma_{\max}} - f\|_{2}^{2}\\
 & &
 +\normalfont{\textrm{pen}}(\gamma) -\frac{2}{n^2}\sum_{k = 1}^{n}\langle
 K_{h_k(\gamma)},K_{h_k(\gamma_{\max})}\rangle_2\\
 & &
 -\left(\normalfont{\textrm{pen}}(\widetilde\gamma_n) -\frac{2}{n^2}\sum_{k = 1}^{n}\langle
 K_{h_k(\widetilde\gamma_n)},K_{h_k(\gamma_{\max})}\rangle_2\right)\\
 & &
 +\frac{\kappa_7}{\varepsilon}\left(\frac{\lambda^2}{n} +\frac{\lambda^3}{n^2h_n(\gamma_{\max})}\right)\\
 & = &
 (1 +\varepsilon)\|\widehat f_{n,\mathbf h_n(\gamma)} - f\|_{2}^{2}
 +\frac{\kappa_6}{\varepsilon}\|f_{n,\gamma_{\max}} - f\|_{2}^{2}\\
 & &
 +\frac{\kappa_7}{\varepsilon}\left(\frac{\lambda^2}{n} +\frac{\lambda^3}{n^2h_n(\gamma_{\max})}\right).
\end{eqnarray*}
This concludes the proof. $\Box$
%


%
\subsubsection{Proof of Lemma \ref{U_statistic_bound}}
Consider
\begin{displaymath}
\Delta_n :=
\{(k,l)\in\mathbb N^2 :
2\leqslant k\leqslant n
\textrm{ and }
1\leqslant l\leqslant k - 1\}.
\end{displaymath}
The $U$-statistic satisfies
\begin{displaymath}
U_n(\gamma,\gamma_{\max}) =
\sum_{k = 2}^{n}\sum_{l < k}
(G_{\gamma,\gamma_{\max}}^{k,l}(X_k,X_l) +
G_{\gamma_{\max},\gamma}^{k,l}(X_k,X_l)),
\end{displaymath}
where
\begin{displaymath}
G_{a,b}^{k,l}(\alpha,\beta) :=
\langle K_{h_k(a)}(\alpha -\cdot) - f_{h_k(a)}
,K_{h_l(b)}(\beta -\cdot) - f_{h_l(b)}
\rangle_2
\end{displaymath}
for every $(k,l)\in\Delta_n$, $a,b\in\{\gamma,\gamma_{\max}\}$ and $(\alpha,\beta)\in\mathbb R^2$.
\\
\\
By Houdr\'e and Reynaud-Bourret \cite{HRB03}, Theorem 3.4, there exists a universal constant $\mathfrak c > 0$ such that
\begin{equation}\label{RH}
 \mathbb P(|U_n(\gamma,\gamma_{\max})|
 \geqslant\mathfrak c(C\sqrt\lambda + D\lambda + B\lambda^{3/2} + A\lambda^2))
 \leqslant 5.54e^{-\lambda}
\end{equation}
where the constants $A$, $B$, $C$ and $D$ will be defined and controlled in the sequel. 
\begin{itemize}
 \item\textbf{The constant $A$.} Consider
 \begin{displaymath}
 A :=\max_{(k,l)\in\Delta_n}\sup_{(\alpha,\beta)\in\mathbb R^2}
 A_{k,l}(\alpha,\beta)
 \end{displaymath}
 with
 \begin{displaymath}
 A_{k,l}(\alpha,\beta) :=
 |G_{\gamma,\gamma_{\max}}^{k,l}(\alpha,\beta) +
 G_{\gamma_{\max},\gamma}^{k,l}(\alpha,\beta)|
 \textrm{ $;$ }
 \forall (k,l)\in\Delta_n\textrm{, }
 \forall (\alpha,\beta)\in\mathbb R^2.
 \end{displaymath}
 For any $(k,l)\in\Delta_n$ and $(\alpha,\beta)\in\mathbb R^2$,
 \begin{eqnarray*}
  A_{k,l}(\alpha,\beta)
  & \leqslant &
  |\langle K_{h_k(\gamma)}(\alpha -\cdot) - f_{h_k(\gamma)},
  K_{h_l(\gamma_{\max})}(\beta -\cdot) - f_{h_l(\gamma_{\max})}\rangle_2|\\
  & &
  + |\langle K_{h_k(\gamma_{\max})}(\alpha -\cdot) - f_{h_k(\gamma_{\max})},
  K_{h_l(\gamma)}(\beta -\cdot) - f_{h_l(\gamma)}\rangle_2|\\
  & \leqslant &
  2(\|K_{h_k(\gamma_{\max})}\|_{\infty} +\|f_{h_k(\gamma_{\max})}\|_{\infty})(\|K\|_1 +\|f_{h_l(\gamma)}\|_1)\\
  & \leqslant &
  8\frac{
  \|K\|_1\|K\|_{\infty}}{h_n(\gamma_{\max})}.
 \end{eqnarray*}
 Therefore,
 \begin{displaymath}
 \frac{A\lambda^2}{n^2}
 \leqslant
 8\frac{\|K\|_1\|K\|_{\infty}}{n^2h_n(\gamma_{\max})}\lambda^2.
 \end{displaymath}
 \item\textbf{The constant $B$.} Consider
 \begin{displaymath}
 B^2 := 
 \max\left\{
 \sup_{\alpha,l}
 \sum_{k = 1}^{l - 1} 
 \mathbb E(|G_{\gamma,\gamma_{\max}}^{k,l}(\alpha,X_l)|^2)
 \textrm{ $;$ }
 \sup_{\alpha,l}
 \sum_{k = l + 1}^{n}
 \mathbb E(
 |G_{\gamma_{\max},\gamma}^{k,l}(\alpha,X_l)|^2)\right\}.
 \end{displaymath}
 For any $(k,l)\in\Delta_n$, $a,b\in\{\gamma,\gamma_{\max}\}$ and $(\alpha,\beta)\in\mathbb R^2$,
 \begin{eqnarray*}
  \mathbb E(G_{a,b}^{k,l}(\alpha,X_l)^2)
  & = &
  \mathbb E(\langle K_{h_k(a)}(\alpha -\cdot) - f_{h_k(a)},
  K_{h_l(b)}(X_l -\cdot) - f_{h_l(b)}\rangle_{2}^{2})\\
  & \leqslant &
  \|K_{h_k(a)}(\alpha -\cdot) - f_{h_k(a)}\|_{2}^{2}
  \mathbb E(\|K_{h_l(b)}(X_l -\cdot) - f_{h_l(b)}\|_{2}^{2})\\
  & \leqslant &
  4\frac{\|K\|_{2}^{2}}{h_k(a)}
  \int_{-\infty}^{\infty}\mathbb E(
  |K_{h_l(b)}(X_l - y) - f_{h_l(b)}(y)|^2)dy\\
  & \leqslant &
  4\frac{\|K\|_{2}^{2}}{h_n(a)}   \|K_{h_l(b)}\|_{2}^{2}
  \leqslant
  4\frac{\|K\|_{2}^{4}}{h_k(a)h_l(b)} \leqslant 
  4\frac{\|K\|_{2}^{4}}{h_k(a)h_n(b)}.
 \end{eqnarray*}
 Then,
 \begin{displaymath}
 B^2 \leqslant 4\frac{\|K\|_{2}^{4}}{h_n(\gamma_{\max})} 
 \sum_{k = 1}^{n}
 \frac{1}{h_k(\gamma)}.
 \end{displaymath}
 Therefore,
 \begin{eqnarray*}
  \frac{B\lambda^{3/2}}{n^2} & \leqslant &
 2\left(\frac{\theta}{3}\right)^{1/2}
 \|K\|_2
 \sqrt{\frac{1}{n^2}
 \sum_{k = 1}^{n}
 \frac{1}{h_k(\gamma)}}
 \times
 \left(\frac{3}{\theta}\right)^{1/2}
 \frac{\|K\|_2}{(n^2h_n(\gamma_{\max}))^{1/2}}\lambda^{3/2}\\
 & \leqslant &
 \frac{\theta\|K\|_{2}^{2}}{3n \,\mathfrak h_n(\gamma)}
 +\frac{3}{\theta}\times
 \frac{\|K\|_{2}^{2}}{ n^2h_n(\gamma_{\max})}\lambda^3.
 \end{eqnarray*}
 \item\textbf{The constant $C$.} Consider
 \begin{displaymath}
 C^2 :=\sum_{(k,l)\in\Delta_n}
 \mathbb E((G_{\gamma,\gamma_{\max}}^{k,l}(X_k,X_l) +
 G_{\gamma_{\max},\gamma}^{k,l}(X_k,X_l))^2).
 \end{displaymath}
  For any $(k,l)\in\Delta_n$ and $a,b\in\{\gamma,\gamma_{\max}\}$,
  \begin{eqnarray*}
   \mathbb E(G_{a,b}^{k,l}(X_k,X_l)^2) & = &
   \mathbb E(\langle K_{h_k(a)}(X_k -\cdot) - f_{h_k(a)},
   K_{h_l(b)}(X_l -\cdot) - f_{h_l(b)}\rangle_{2}^{2})\\
   & \leqslant &
   \kappa_1(\mathbb E(\langle K_{h_k(a)}(X_k -\cdot),K_{h_l(b)}(X_l -\cdot)\rangle_{2}^{2})\\
   & &
   +\|f_{h_l(b)}\|_{\infty}^{2}\|K\|_{1}^{2}
   +\|f_{h_k(a)}\|_{\infty}^{2}\|K\|_{1}^{2}
   +\|f_{h_k(a)}\|_{\infty}^{2}\|f_{h_l(b)}\|_{1}^{2})\\
   & \leqslant &
   \kappa_2\left(\mathbb E\left(\left|
   \int_{-\infty}^{\infty}K_{h_k(a)}(X_k - x)K_{h_l(b)}(X_l - x)dx\right|^2\right) +
   \|f\|_{\infty}^{2}\|K\|_{1}^{4}
   \right).
  \end{eqnarray*}
  Moreover,
  \begin{eqnarray*}
   \mathbb E\left(\left|
   \int_{-\infty}^{\infty}K_{h_k(a)}(X_k - x)K_{h_l(b)}(X_l - x)dx\right|^2\right)
   & \leqslant &
   \frac{\|K\|_{1}^{2}\|K\|_{2}^{2}\|f\|_{\infty}}{h_k(a)}.
  \end{eqnarray*}
  Then,
  \begin{displaymath}
  C\leqslant
  \kappa_3\sqrt{n}\|K\|_1\|f\|_{\infty}^{1/2}
  \left(\|K\|_2
  \sqrt{\sum_{k = 1}^{n}\frac{1}{h_k(\gamma)}}
  +\|K\|_1\|f\|_{\infty}^{1/2}\right).
  \end{displaymath}
  Therefore, since $\lambda\in [1,\infty[$,
  \begin{displaymath}
  \frac{C\lambda^{1/2}}{n^2} 
  \leqslant
  \frac{\theta\|K\|_{2}^{2}}{3n \,\mathfrak h_n(\gamma)} +
  \kappa_4\frac{\|K\|_{1}^{2}\|f\|_{\infty}}{\theta n}\lambda.
  \end{displaymath}
  \item\textbf{The constant $D$.} Consider
  \begin{displaymath}
  D :=\sup_{(a,b)\in\mathcal S}
  \sum_{k = 2}^{n}\sum_{l = 1}^{k - 1}
  \mathbb E((G_{\gamma,\gamma_{\max}}^{k,l}(X_k,X_l) +
  G_{\gamma_{\max},\gamma}^{k,l}(X_k,X_l))a_k(X_k)b_l(X_l)),
  \end{displaymath}
  where
  \begin{displaymath}
  \mathcal S :=
  \left\{(a,b) :
  \sum_{k = 2}^{n}
  \mathbb E(a_k(X_k)^2)\leqslant 1
  \textrm{ and }
  \sum_{l = 1}^{n - 1}
  \mathbb E(b_l(X_l)^2)\leqslant 1
  \right\}.
  \end{displaymath}
  For any $(a,b)\in\mathcal S$,
  \begin{displaymath}
  \sum_{k = 2}^{n}
  \sum_{l = 1}^{k - 1}\mathbb E(G_{\gamma,\gamma_{\max}}^{k,l}(X_k,X_l)a_k(X_k)b_l(X_l))
  \leqslant
  D_2(a,b)\sup_{x\in\mathbb R}D_1(a,b,x)
  \end{displaymath}
  with
  \begin{eqnarray*}
   D_1(a,b,x) & := &
   \sum_{k = 2}^{n}\mathbb E(
   |a_k(X_k)(K_{h_k(\gamma)}(X_k - x) - f_{h_k(\gamma)}(x))|)\\
   & \leqslant &
   \mathbb E\left[\left|\sum_{k = 2}^{n}
   a_k(X_k)^2\right|^{1/2} 
   \left|\sum_{k=1}^n (K_{h_k(\gamma)}(X_k - x) - f_{h_k(\gamma)}
   (x))^2\right|^{1/2}\right]\\
   & \leqslant &
   \mathbb E\left(\sum_{k = 2}^{n}
   a_k(X_k)^2\right)^{1/2}\left|\sum_{k = 1}^{n}
   \mathbb E(|K_{h_k(\gamma)}(X_k - x) - 
   f_{h_k(\gamma)}(x)|^2)\right|^{1/2}\\
   & \leqslant &
   \left|\sum_{k = 1}^{n}
   \mathbb E(K_{h_k(\gamma)}(X_k - x)^2)\right|^{1/2}
   \leqslant
   \left|\sum_{k = 1}^{n}
   \frac{\|K\|_{2}^{2}\|f\|_{\infty}}{h_k(\gamma)}\right|^{1/2}
  \end{eqnarray*}
  and
  \begin{eqnarray*}
   D_2(a,b) & := &
   \sum_{l = 1}^{n - 1}\mathbb E\left(|b_l(X_l)|\int_{-\infty}^{\infty}|K_{h_l(\gamma_{\max})}(X_l - x) - f_{h_l(\gamma_{\max})}(x)|dx\right)\\
   & \leqslant &
   2\|K\|_1\sum_{l = 1}^{n - 1}\mathbb E(|b_l(X_l)|)
   \leqslant
   2\|K\|_1\sqrt n\left|
   \sum_{l = 1}^{n - 1}\mathbb E(b_l(X_l)^2)\right|^{1/2}
   \leqslant 2\sqrt n\|K\|_1.
  \end{eqnarray*}
  Then, 
  \begin{displaymath}
  D\leqslant 2\sqrt{n}\|f\|_{\infty}^{1/2}\|K\|_1\|K\|_2
  \left|\sum_{k = 1}^{n} \frac 1{h_k(\gamma)}\right|^{1/2}.
  \end{displaymath}
  Therefore, 
  \begin{displaymath}
  \frac{D\lambda}{n^2}
  \leqslant
  \frac{\theta\|K\|_{2}^{2}}{3n \,\mathfrak h_n(\gamma)} +
  \frac{12}{\theta}\times\frac{\|K\|_{1}^{2}\|f\|_{\infty}}{n}\lambda^2.
  \end{displaymath} 
\end{itemize}
Plugging the bounds obtained for $A, B, C, D$ in Inequality (\ref{RH}) gives the announced result and ends the proof. $\Box$
%


%
\subsubsection{Proof of Lemma \ref{bound_V}}
For any $\gamma'\in\Gamma_n$,
\begin{displaymath}
V_n(\gamma,\gamma') =
\frac{1}{n}
\sum_{k = 1}^{n}(g_{\gamma'}(h_k(\gamma),X_k) -\mathbb E(g_{\gamma'}(h_k(\gamma),X_k)))
\end{displaymath}
where, for any $k\in\{1, \dots, n\}$,
\begin{displaymath}
g_{\gamma'}(h_k(\gamma),X_k) :=
\langle K_{h_k(\gamma)}(X_k -\cdot),f_{n,\gamma'} - f\rangle_2.
\end{displaymath}
Indeed,
\begin{displaymath}
\mathbb E(g_{\gamma'}(h_k(\gamma),X_k)) =
\langle\mathbb E(K_{h_k(\gamma)}(X_k -\cdot)),f_{n,\gamma'} - f\rangle_2 =
\langle f_{h_k(\gamma)},f_{n,\gamma'} - f\rangle_2.
\end{displaymath}
In order to apply Bernstein's inequality to $g_{\gamma'}(h_1(\gamma),X_1),\dots,g_{\gamma'}(h_n(\gamma),X_n)$, let us find suitable controls of
\begin{displaymath}
c_{\gamma'} :=\frac{\|g_{\gamma'}\|_{\infty}}{3}
\textrm{ and }
\upsilon_n(\gamma,\gamma') :=
\frac{1}{n}\sum_{k = 1}^{n}\mathbb E(g_{\gamma'}(h_k(\gamma),X_k)^2).
\end{displaymath}
On the one hand,
\begin{eqnarray*}
 c_{\gamma'} & = &
 \frac{1}{3}\sup_{h > 0,x\in\mathbb R}
 |\langle K_h(x -\cdot),f_{n,\gamma'} - f\rangle_2|\\
 & \leqslant &
 \frac{1}{3}  \sup_{h > 0,x\in\mathbb R}
 \|K_h(x -\cdot)\|_1 \|f_{n,\gamma'} - f\|_{\infty} \\
 & \leqslant &
 \frac{1}{3}
 \|K\|_1\max_{k\in\{1, \dots, n\}}\|K_{h_k(\gamma')}\ast f - f\|_{\infty}\\
 & \leqslant &
 \frac{1}{3}
 \|K\|_1(1 +\|K\|_1)\|f\|_{\infty}
 \leqslant
 \frac 23 \|K\|_1^2 \|f\|_\infty,
\end{eqnarray*}
as $\|K\|_1\geqslant 1$. On the other hand,
\begin{eqnarray*}
 \upsilon_n(\gamma,\gamma') & = &
 \frac{1}{n}\sum_{k = 1}^{n}
 \int_{-\infty}^{\infty}
 \left(\int_{-\infty}^{\infty}K_{h_k(\gamma)}(y - x)(f_{n,\gamma'}(x) - f(x))dx\right)^2f(y)dy\\
 & \leqslant &
 \|f\|_{\infty}\max_{k\in\{1, \dots, n\}}
 \|K_{h_k(\gamma)}\ast(f_{n,\gamma'} - f)\|_{2}^{2}
 \leqslant
 \|f\|_{\infty}\|K\|_{1}^{2}\|f_{n,\gamma'} - f\|_{2}^{2}.
\end{eqnarray*}
Then, by Bernstein's inequality, with probability larger than $1 - 2e^{-\lambda}$,
\begin{eqnarray*}
 |V_n(\gamma,\gamma')| & \leqslant &
 \sqrt{\frac{2\lambda}{n}\upsilon_n(\gamma,\gamma')} +
 \frac{c_{\gamma'}\lambda}{n}\\
 & \leqslant &
 \sqrt{\frac{2\lambda}{n}\|f\|_{\infty}\|K\|_{1}^{2}\|f_{n,\gamma'} - f\|_{2}^{2}} +
 \frac{\lambda\|K\|_1(1 +\|K\|_1)\|f\|_{\infty}}{3n}\\
 & \leqslant &
 \theta\|f_{n,\gamma'} - f\|_{2}^{2} +\frac{c\lambda}{\theta n},
\end{eqnarray*}
with $c = 7/6\|f\|_\infty\|K\|_1^2$. This is the announced inequality. $\Box$
%


%
\subsubsection{Proof of Lemma \ref{Lerasle_type_inequality}}
First of all,
\begin{displaymath}
\|f_{n,\gamma} - f\|_{2}^{2} =
\|\widehat f_{n,\mathbf h_n(\gamma)} - f\|_{2}^{2} -
\|\widehat f_{n,\mathbf h_n(\gamma)} - f_{n,\gamma}\|_{2}^{2} -
2V_n(\gamma,\gamma).
\end{displaymath}
Then, by Lemma \ref{bound_V}, with probability larger than $1 - 2e^{-\lambda}$,
\begin{equation}\label{Lerasle_type_inequality_1}
(1 -\theta)\|f_{n,\gamma} - f\|_{2}^{2} +
\frac{\|K\|_{2}^{2}}{n \, \mathfrak h_n(\gamma)}
\leqslant
\|\widehat f_{n,\mathbf h_n(\gamma)} - f\|_{2}^{2}
+\Lambda_n(\gamma)
+\frac{\kappa_1\lambda}{\theta n}
\end{equation}
where
\begin{eqnarray*}
 \Lambda_n(\gamma) & := &
 \left|\frac{\|K\|_{2}^{2}}{n \, \mathfrak h_n(\gamma)} -
 \|\widehat f_{n,\mathbf h_n(\gamma)} - f_{n,\gamma}\|_{2}^{2}\right|\\
 & = &
 \left|\frac{U_n(\gamma,\gamma)}{n^2}
 +\frac{W_n(\gamma)}{n}
 -\frac{1}{n^2}\sum_{k = 1}^{n}\|f_{h_k(\gamma)}\|_{2}^{2}\right|
\end{eqnarray*}
and
\begin{displaymath}
W_n(\gamma) :=
\frac{1}{n}\sum_{k = 1}^{n}(
Y_k(\gamma) -\mathbb E(Y_k(\gamma)))
\end{displaymath}
with, for any $k\in\{1, \dots, n\}$,
\begin{displaymath}
Y_k(\gamma) :=
\|K_{h_k(\gamma)}(X_k -\cdot) - f_{h_k(\gamma)}\|_{2}^{2}
\end{displaymath}
and
\begin{eqnarray*}
 \mathbb E(Y_k(\gamma)) & = &
 \mathbb E(\|K_{h_k(\gamma)}(X_k -\cdot)\|_{2}^{2}) +\|f_{h_k(\gamma)}\|_{2}^{2}
 -2\langle\mathbb E(K_{h_k(\gamma)}(X_k -\cdot)),f_{h_k(\gamma)}\rangle_2\\
 & = &
 \frac{\|K\|_{2}^{2}}{h_k(\gamma)} -\|f_{h_k(\gamma)}\|_{2}^{2}.
\end{eqnarray*}
Since
\begin{displaymath}
|Y_k(\gamma)|\leqslant
4\|K_{h_k(\gamma)}\|_{2}^{2}\leqslant
M_n(\gamma) := 4\frac{\|K\|_{2}^{2}}{h_n(\gamma_{\max})}
\end{displaymath}
and
\begin{displaymath}
\mathbb E(Y_k(\gamma)^2)\leqslant
M_n(\gamma)\mathbb E(Y_k(\gamma))
\leqslant
4\frac{\|K\|_{2}^{4}}{h_n(\gamma_{\max})h_k(\gamma)},
\end{displaymath}
by Bernstein's inequality, with probability larger than $1 - 2e^{-\lambda}$,
\begin{eqnarray*}
 |W_n(\gamma)|
 & \leqslant &
 2\sqrt{\frac{4\|K\|_{2}^{2}\lambda}{\theta n h_n(\gamma_{\max})}\times\frac{\theta}{2}\frac{\|K\|_{2}^{2}}{\mathfrak h_n(\gamma)}}
 +\frac{4\|K\|_{2}^{2}\lambda}{3nh_n(\gamma_{\max})}\\
 & \leqslant &
 \frac{\theta}{2}\frac{\|K\|_{2}^{2}}{\mathfrak h_n(\gamma)} +
 \frac{\kappa_2\lambda}{\theta n h_n(\gamma_{\max})}.
\end{eqnarray*}
Moreover, by Jensen's inequality,
\begin{eqnarray*}
 \|f_{h_k(\gamma)}\|_{2}^{2} & = &
 \int_{-\infty}^{\infty}\left(\int_{-\infty}^{\infty}
 f(x + h_k(\gamma)y)K(y)dy\right)^2dx\\
 & \leqslant &
 \|K\|_{1}^{2}\int_{-\infty}^{\infty}\int_{-\infty}^{\infty}f(x + h_k(\gamma)y)^2\frac{|K(y)|}{\|K\|_1}dydx
 \leqslant
 \|f\|_{\infty}\|K\|_{1}^{2}.
\end{eqnarray*}
Then, by Lemma \ref{U_statistic_bound}, with probability larger than $1 - \kappa_3e^{-\lambda}$,
\begin{eqnarray*}
 \Lambda_n(\gamma) & \leqslant &
 \theta\frac{\|K\|_{2}^{2}}{n \,\mathfrak h_n(\gamma)} +
 \kappa_4\left(\frac{\lambda^2}{\theta n} +\frac{\lambda^3}{\theta n^2h_n(\gamma_{\max})}\right).
\end{eqnarray*}
Therefore, by Inequality (\ref{Lerasle_type_inequality_1}), with probability larger than $1 -\kappa_5e^{-\lambda}$,
\begin{displaymath}
\|f_{n,\gamma} - f\|_{2}^{2} +
\frac{\|K\|_{2}^{2}}{n \,\mathfrak h_n(\gamma)}
\leqslant
\frac{1}{1 -\theta}
\|\widehat f_{n,\mathbf h_n(\gamma)} - f\|_{2}^{2}+
\frac{\kappa_6}{\theta(1 -\theta)}
\left(\frac{\lambda^2}{n} +\frac{\lambda^3}{n^2h_n(\gamma_{\max})}\right).
\end{displaymath}
This is the announced result if we set $1+\varepsilon=1/(1-\theta)$, which gives 
$1/[\theta (1-\theta)]=(1+\varepsilon)^2/\varepsilon$. $\Box$

%
%

%


%

%
\end{document}